\documentclass[11pt]{amsart}
\usepackage{amsmath,amssymb,amsthm,graphicx, url, verbatim, float,hyperref,ulem} 
\usepackage{enumerate}
\usepackage{xcolor, relsize}

\newtheorem{lemma}{Lemma}[section]
\newtheorem{proposition}[lemma]{Proposition}
\newtheorem{theorem}[lemma]{Theorem}

\newtheorem{corollary}[lemma]{Corollary}
\newtheorem{question}[lemma]{Question}
\newtheorem{conjecture}[lemma]{Conjecture}

\newcommand{\bcon}{\begin{conjecture}}
\newcommand{\econ}{\end{conjecture}}
\newcommand{\bcor}{\begin{corollary}}
\newcommand{\ecor}{\end{corollary}}
\newcommand{\bdf}{\begin{definition}}
\newcommand{\edf}{\end{definition}}
\newcommand{\benu}{\begin{enumerate}}
\newcommand{\eenu}{\end{enumerate}}
\newcommand{\beq}{\begin{equation}}
\newcommand{\eeq}{\end{equation}}
\newcommand{\bexa}{\begin{example}}
\newcommand{\eexa}{\end{example}}
\newcommand{\bexe}{\begin{exercise}}
\newcommand{\eexe}{\end{exercise}}
\newcommand{\bfac}{\begin{fact}}
\newcommand{\efac}{\end{fact}}
\newcommand{\bite}{\begin{itemize}}
\newcommand{\eite}{\end{itemize}}
\newcommand{\blem}{\begin{lemma}}
\newcommand{\elem}{\end{lemma}}
\newcommand{\bmat}{\begin{matrix}}
\newcommand{\emat}{\end{matrix}}
\newcommand{\bprb}{\begin{problem}}
\newcommand{\eprb}{\end{problem}}
\newcommand{\bpro}{\begin{proposition}}
\newcommand{\epro}{\end{proposition}}

\newcommand{\bque}{\begin{question}}
\newcommand{\eque}{\end{question}}
\newcommand{\brem}{\begin{remark}}
\newcommand{\erem}{\end{remark}}
\newcommand{\bthm}{\begin{theorem}}
\newcommand{\ethm}{\end{theorem}}
\newtheorem*{namedtheorem}{\theoremname}

\newcommand{\bpr}{\begin{proof}}
\newcommand{\epr}{\end{proof}}

\theoremstyle{definition}
\newtheorem{definition}[lemma]{Definition}
\newtheorem{remark}[lemma]{Remark}
\newtheorem{example}[lemma]{Example}

\newcommand{\theoremname}{testing}

\newcommand{\m}{\mu}
\newcommand{\lo}{\lambda}
\newcommand{\p}{\partial}

\newcommand{\la}{\langle}
\newcommand{\ra}{\rangle}

\def\cT{\mathcal T}
\def\cH{{\mathcal H}}
\def\P{\mathbb P}
\def\t{\tau_{n,q}}
\newcommand{\F}{\mathbb{F}}
\newcommand{\Z}{\mathbb{Z}}
\newcommand{\R}{\mathbb{R}}
\newcommand{\Q}{\mathbb{Q}}
\newcommand{\C}{\mathbb{C}}
\newcommand{\N}{\mathbb{N}}
\newcommand{\Tr}{\mathrm{Tr}}

\newcommand{\slC}{\mathrm{SL}_2(\mathbb{C})}

\newcommand{\vs}{\vspace*{.1in}}

\title[Kauffman bracket skein modules of small 3-manifolds]{ Kauffman bracket skein modules of small 3-manifolds}
\author{Renaud Detcherry}
\date{} 

\address{Université Bourgogne Europe, CNRS, IMB UMR 5584, F-21000 Dijon, France}
\email{renaud.detcherry@u-bourgogne.fr}
\author{Efstratia Kalfagianni}
\address{Department of Mathematics, Michigan State University, East
Lansing, MI, 48824, USA}
\email{kalfagia@math.msu.edu}

\author{Adam S. Sikora}
\address{Department of Mathematics, University at Buffalo,
Buffalo, NY, 14260, USA}
\email{asikora@buffalo.edu}

\thanks{2020 {\em Mathematics Classification:} 
Primary: 
57K31, 
Secondary: 57K16.  
}

\def\pmo{{\pm 1}}

\begin{document}

\begin{abstract}
The proof of Witten's finiteness conjecture established that the Kauffman bracket skein modules of closed $3$-manifolds are finitely generated over $\Q(A)$. In this paper, we develop a novel method for computing these  skein modules. 

We show that if the skein module $S(M,\Q[A^\pmo])$ of $M$ is tame  (e.g.  finitely generated over $\Q[A^{\pm 1}]$), and the $SL(2,
\C)$-character scheme is reduced,
 then the dimension $\dim_{\Q(A)}\, S(M, \Q(A))$ is the number of closed points in this character scheme. 
This, in particular,  verifies a conjecture in the literature relating  $\dim_{\Q(A)}\, S(M, \Q(A))$ to the Abouzaid-Manolescu 
$SL(2,\C)$-Floer theoretic invariants, for infinite families of 3-manifolds. 

We prove a criterion for reducedness of character varieties of closed $3$-manifolds and use it to
 compute the skein modules of Dehn fillings of $(2,2n+1)$-torus knots and of the figure-eight knot. The later family  gives the
first instance of computations of skein modules for closed hyperbolic 3-manifolds.

We also prove that the skein modules of rational homology spheres have dimension at least $1$ over $\Q(A)$.
\end{abstract}
\maketitle

\section{Introduction}
\label{sec:intro}

\def\cX{\mathcal X}
\def\cR{\mathcal R}

Throughout the paper, $M$ will denote an oriented $3$-manifold and $S(M,R)$ its Kauffman bracket skein module  with coefficients in a commutative ring $R$ with a distinguished invertible element $A\in R$.
The most frequent choice of ring of coefficients in this article will be $R=\Q(A),$ in which case we simply write $S(M)$ for $S(M,\Q(A)).$

Skein modules were originally introduced by Przytycki \cite{Przytycki} and Turaev \cite{Turaev}. 
They received particular attention in the recent years \cite{BW16, FKL19, GJS19a, GJS19}, due to their connections  to quantum groups, cluster algebras, quantum field theories, and other areas of mathematics and physics. 
Witten conjectured and Gunningham, Jordan, and Safronov proved  \cite{GJS19}, that the skein module 
$S(M)$ is finite dimensional for any closed 3-manifold $M$.  However, their work doesn't offer an effective method for computing the dimension $\dim_{\Q(A)} S(M)$.  

We prove that, under certain conditions, the dimension
$\dim_{\Q(A)}\, S(M)$ coincides with the number of $SL(2,\C)$-representations of $\pi_1(M)$, up to conjugation. We make this statement precise below.

Given a connected manifold $M$, let
$$\cX(M):=\mathrm{Hom}(\pi_1(M),\slC)/\hspace*{-.05in}/\slC,$$
be the $\slC$-character scheme of $M$ and let $\C[\cX(M)]$ be the coordinate ring of $\cX(M)$. 
We will consider  $\cX(M)$ as a scheme over $\C$, as defined, for example in \cite{LM85,BH95}.  In this setting $\C[\cX(M)]$ is the algebra of global sections of the structure sheaf of $\cX(M)$.
As a scheme $\cX(M)$ may be unreduced, that is $\C[\cX(M)]$ may have a nontrivial nil-radical. 
Examples of $3$-manifolds with non-reduced character schemes can be found in \cite{KM17}.

We denote by $X(M)$ the algebraic set underlying $\cX(M)$. Hence, 
$\C[X(M)]=\C[\cX(M)]/\sqrt{0},$ 
where $\sqrt{0}$ is the nil-radical of $\C[\cX(M)]$.
We will write $|X(M)|$ for the cardinality of
$X(M)$, which can be finite or infinite.
It is known that two representations $\rho, \rho': \pi_1(M)\to SL(2,\C)$ are identified in $X(M)$ if and only if their traces coincide, $tr\rho = tr\rho'.$ Hence, $X(M)$ can be considered as the set of all $SL(2,\C)$-characters of $\pi_1(M)$.

We will say that a $\Q[A^{\pm 1}]$-module $S$ is \underline{tame} if it is a direct sum of cyclic $\Q[A^{\pm 1}]$-modules and $S$ does not contain $\Q[A^{\pm 1}]/(\phi_{2N})$ as a submodule, for at least one odd $N$, where $\phi_{2N}$ is the $2N$-th cyclotomic polynomial. In particular, every finitely generated $\Q[A^{\pm 1}]$-module is tame.

One of the main results of the paper is the following:
 
\begin{theorem}\label{t.main1-i} If $M$ is a closed
$3$-manifold with tame $S(M, \Q[A^{\pm 1}])$, then
$$|X(M)| \leq \dim_{\Q(A)}  S(M) \leq \dim_{\C}\, \C[\cX(M)].$$
In particular, if $\cX(M)$ is reduced, then 
 $\dim _{\Q(A)}S(M)=|X(M)|$. Furthermore, $S(M, \Q[A^{\pm 1}])$ has no $(A+1)$-torsion in that case.
 \end{theorem}
 
The proof of Theorem \ref{t.main1-i} relies on major  recent advances on the structure of surface skein modules at roots of unity by
 Bonahon-Wong \cite{BW16}, Ganev-Jordan-Safronov \cite{GJS19a} and  Frohman-Kania-Bartoszy\'nska-L\^{e} \cite{FKL19}.
 It also uses a result of \cite{Abelian} whose proof relies on the theory of the non-semisimple $sl_2$-quantum invariants of 3-manifolds developed by Constantino, Geer and Patureau-Mirand \cite{NSINV}.

Theorem \ref{t.main1-i} also  provides new information about skein modules $S(M)$  for 3-manifolds with infinite $X(M)$.  Indeed, since $\dim_{\Q(A)}S(M)$ is finite,
 $S(M, \Q[A^{\pm 1}])$ is not tame for any closed $M$ with infinite $X(M)$. 
In fact, in Proposition \ref{prop:infiniteX}, we are able to extract more detailed information about the structure of $S(M)$ in this case.
By Culler-Shallen theory \cite{CullerShalen}, such manifolds contain incompressible surfaces.
It is worth noting that if Conjecture (E) of Problem 1.92 of \cite{Kirby} holds then it implies the converse: if $M$ contains no incompressible surfaces, then $S(M, \Q[A^{\pm 1}]))$ is tame.

\begin{remark}\label{r.GJSconj} In \cite[Conjecture D]{GS23} and \cite[Section 6.3]{GJS19}, it is conjectured that $\dim_{\Q(A)} S(M)$ is equal to
the dimension of the zero degree part of the  Abouzaid-Manolescu 
homology $HP^{\bullet}_{\#}(M)$
 \cite{AM20}. It follows from results of \cite{AM20} that
 if $M$ is a $\Z$-homology sphere and $\cX(M)$ is finite and reduced, this later dimension is  $|X(M)|$, cf. our Proposition \ref{p.HP}.
 Therefore, Theorem \ref{t.main1-i} verifies the conjecture of \cite{GJS19} under these assumptions and that of $S(M, \Q[A^{\pm 1}])$ being tame. In Section  \ref{sec:questions}, we provide more details and we discuss families of $3$-manifolds for which we are able to verify the conjecture.
 \end{remark}



\subsection{Tameness and reducedness under Dehn filling} 
 Questions about reducedness of $\cX(M)$ and about  tameness of $S(M)$ are very difficult in general.
However, we are able to answer those questions for 3-manifolds  $E_K(p/q)$ obtained by Dehn fillings on the  figure-eight knot $K=4_1$, and on
 the $(2,2n+1)$-torus knots $K=T_{2,2n+1}$.
  Here $E_K$ denotes the complement of an open tubular neighborhood of $K$ in $S^3$, and $E_K(p/q)$ denotes its Dehn filling with slope $p/q$, for coprime integers $p,q$. 
We prove the following:

\begin{theorem}\label{t.red-tame} 

(a) For  $M:=E_{T_{(2,2n+1)}}(p/q)$, $n\in \Z,$ the skein module $S(M, \Q[A^{\pm 1}])$ is finitely generated (and, hence, tame) for $p/q \notin \lbrace 0,\ 4n+2 \rbrace$. The scheme $\cX(M)$ is finite and reduced for all slopes $p/q$, where $p$ is either not divisible by $4$ or coprime with $2n+1$.\\
\hspace*{.1in}
(b) For $M:=E_{4_1}(p/q)$, the skein $S(M, \Q[A^{\pm 1}])$ is finitely generated for $p/q\notin\lbrace 0,\pm 4 \rbrace$. The scheme $\cX(M)$ is finite and reduced for all but finitely many $p/q$, including all slopes with $p=1$.
\end{theorem}

The proof of Theorem \ref{t.red-tame} combines different techniques and is spread out over Sections  \ref{sec:examples}, \ref{sec:reduceness}, \ref{sec:reduceness_41} and \ref{ss.torusknot}.

Note that for  3-manifolds $M=E_K(p/q)$, where $K$ is the figure-eight or a $(2, 2n+1)$-torus knot, our exceptions for $S(M, \Q[A^{\pm 1}])$ being finitely generated
are exactly these where $M$ is Haken. We postulate that this is the case for all closed manifolds in Conjecture \ref{c.non-Haken} below.  In \cite{Seifert} we have verified Conjecture
\ref{c.non-Haken} for Seifert fibered 3-manifolds.

\def\cal{\mathcal}
\def\cC{\cal C}

Given a knot $K$, the natural embedding of  the torus boundary $\p E_K$ into $E_K$ induces a map $r: \cX(E_K)\to X(\p E_K)$. A choice a meridian $m$ and longitude $l$ in $\pi_1(\p E_K)$ identifies $X(\p E_K)$ with the quotient of $\C^*\times\C^*$ by the involution $\tau(\mu,\lambda)=(\mu^{-1},\lambda^{-1}).$ 
The A-polynomial of $K$, denoted by  $A_K(\mu, \lo)$, 
describes the $r$ image of the non-abelian components of $\cX(E_K)$ in $X(\p E_K)$ lifted to $\C^*\times\C^*$, \cite{A-poly}. 

In Section \ref{sec:reduceness}, we prove a stronger version of the following criterion for reducedness of character varieties of Dehn surgeries
on knots:

\begin{theorem} \label{t.reduced-i}
Consider a knot $K$ and a slope $p/q\in \Q$ such that each closed point $\chi\in \cX(E_K(p/q))$ belongs to a unique irreducible component of $\cX(E_K)$, denoted by $\cC_\chi$.
Suppose furthermore that for any such character $\chi$,
\begin{enumerate}[(a)]
\item the map $r$ restricted to an open neighborhood of $\chi$ in $\cC_\chi$ is isomorphism
\footnote{The isomorphism is in the sense of algebraic varieties and, in particular,  implies that 
$\cC_\chi$ is reduced at $\chi$.} 
 onto its image, 
\item the polynomial  $A_{K}(x^{-q},x^{p})$  has no multiple roots, except possibly $\pm 1$, 
\item if $(\chi(l),\chi(m))\in \{\pm 2\}^2$ then $\chi$ is abelian.
\end{enumerate}
Then $\cX(E_K(p/q))$ is finite and reduced.
\end{theorem}

This criterion is partially inspired by results of Charles and March\'e \cite{ChM} and March\'e and Maurin \cite{MM22}. Its utility  is underlined by the fact that verifying reducedness of character varieties of 3-manifolds is hard in general. For instance, it is not known whether character varieties of knot complements are always reduced.


\subsection{Skein modules of surgeries on $4_1$ and $T_{(2,2n+1)}$.}
For the knots  discussed in Theorem \ref{t.red-tame},
our methods allow to compute the dimensions of modules $S(E_{K}(p/q))$ for infinitely many slopes $p/q$. To state our result,
for coprime $p, q\in \Z$, define

$$d_{4_1}(p/q):= {1\over2} (|4q+p| + |4q-p|)- \delta_{2\nmid p},$$
where $\delta_{2\nmid p}=1$ for odd $p$, and 0 otherwise.
Furthermore, let
$$\tau_{n,p,q}:= |p/2-(2n+1) q|- \delta_{2\nmid p}/2\ \text{and}\ d_{T_{(2,2n+1)}}(p/q) := \tau_{n,p,q}\cdot n.$$

Using Theorems \ref{t.main1-i}, \ref{t.red-tame} and a direct analysis of character varieties for surgeries on $4_1$ and $T_{(2,2n+1)}$,  partially relying on \cite{BC06}, we prove the following:

\begin{theorem}\label{t.dimensions} We have
$$\dim_{\Q(A)} S(E_{K}(p/q))=|X(E_{K}(p/q))| = d_K(p/q) +1+\left\lfloor\frac{|p|}{2}\right\rfloor,$$
for
(a) $K=4_1$ and for all but finitely many $p/q$, including all slopes with $p=1$, and \\
(b) $K=T_{(2,2n+1)}$ and all $n\in \Z$ and all slopes $p/q\ne 4n+2$, where $p$ is either not divisible by $4$ or coprime with $2n+1$.
\end{theorem}

Even though skein modules  have been around for more than three decades,  for prime closed $3$-manifolds with coefficients in $\Q(A)$, were only computed for a few Seifert manifolds: lens spaces and $S^2\times S^1$ \cite{HP93, HP95}, the quaternionic manifold \cite{GH07}, some prism manifolds \cite{Mro11a}, trivial $S^1$-bundles over surfaces \cite{GM19, DW}, and the mapping tori of the $2$-torus \cite{Kin}.  
In this paper we construct bases of skein modules over $\Q(A)$ of two new infinite families of $3$-manifolds,
obtained on surgeries on the figure-eight and the knots $T_{(2,2n+1)}$.
The former  family provides the first examples 
of closed hyperbolic manifolds where any
skein modules have been understood.

To give some more detail of our process, recall that by Przytycki-Sikora \cite{PS00}, the skein module
$S_{-1}(M):=S(M, \Z[A^{\pm 1}])\otimes_{\Z[A^{\pm 1}]} \C,$
where $\C$ is identified with the $\Z[A^{\pm 1}]$-module given by $A=-1$,
has a natural structure 
 of $\C$-algebra isomorphic with the coordinate ring $\C[\cX(M)]$ of $\cX(M)$. 
 (A version of this result ``up to nilpotents'' was proved independently in \cite{Bullock}.)
In general,
 there are algorithmic methods to find bases of $S_{-1}(M)\cong \C[X(M)]$.
 When $S(M, \Q[A^{\pm 1}])$  is tame and $\cX(M)$ is reduced, our Proposition \ref{p.basis}  allows to lift those bases to bases of  $S(M)$. 
 We use direct algebraic arguments to give explicit bases for $\C[X(E_{K}(p/q))]$, where $K$ is either $4_1$ or $T_{(2,2n+1)}$ in Theorems \ref{t.basisf8} and \ref{t.torusb}.
 Then we use
  Theorems \ref{t.main1-i}, \ref{t.red-tame}, \ref{t.reduced} and Proposition \ref{p.basis}  to lift these bases to $S(E_{K}(p/q))$.
  For details the reader is referred to Sections 7 and 8.

In  \cite{Seifert} we undertake a more systematic study of  skein modules for Seifert fibered 3-manifolds. In particular we compute the dimension $dim_{\Q(A)} S(M)$
and verify Proposition \ref{p.HP} for all non-Haken Seiferert fibered homology spheres.


\subsection{Non-triviality of skein modules}
At this writing the answer to the basic question of whether there exists a 3-manifold with $S(M)=0$ is not known. We propose the following:
\begin{conjecture}\label{question:nontriviality}
For any orientable 3-manifold $M$ we have $$\dim_{\Q(A)}S(M) \geq 1.$$
\end{conjecture}

By an application of Gilmer-Masbaum's evaluation map \cite{GM19} for skein modules with $\Q(A)$-coefficients and a theorem of Murakami about the $\mathrm{SO}(3)$-Reshetikhin-Turaev invariants of $\Q$-homology spheres, we prove that Conjecture \ref{question:nontriviality} holds for $\Q$- homology spheres.

\begin{theorem}\label{thm:nontriviality} Let $M$ be a rational homology sphere, and $\emptyset$ be the empty link in $M$. Then $\emptyset \neq 0$ in $S(M)$ and, hence,
$\dim_{\Q(A)}S(M) \geq 1$.
\end{theorem}

We generalize Theorem \ref{thm:nontriviality}  to manifolds with boundary in Corollary \ref{cor:Kauffmanbracket2}.


\subsection{Outline of contents} 
Sections \ref{sec:dimension}  and \ref{sec:main-thm-proof} are devoted to the proof of Theorem \ref{t.main1-i}, while Section \ref{sec:examples}, \ref{sec:reduceness_41}, \ref{sec:bases_41} and \ref{ss.torusknot} focus on applications to Dehn-fillings on $4_1$ and $T_{(2,2n+1)}$. Those two parts can be read largely independently, though the second part refers to some of the statements in Section \ref{sec:dimension}.
Section \ref{sec:reduceness} contains the proof of  Theorem \ref{t.reduced}, and some corollaries of it,  and may be of independent interest. Sections \ref{sec:reduceness_41} and \ref{sec:bases_41} study character varieties of the  Dehn surgeries on $4_1$, with a focus on their reducedness and a construction of bases of the coordinate rings of their character varieties. Section \ref{ss.torusknot} achieves the same for $T_{(2,2n+1)}$. These three sections may also be of independent interest, providing an extensive study of the character varieties of those Dehn fillings, and the methods should be applicable to all 2-bridge knots.

Section \ref{sec:non_triviality}, which is independent from the rest of the paper,
studies the non-triviality of skein modules, proving Theorem \ref{thm:nontriviality}.
Section \ref{sec:questions} concludes the paper with some open questions and remarks.
\vspace*{.1in}

\textbf{Acknowledgements:} This work started during   the conference ``Quantum Topology and Geometry conference in the honor of Vladimir Turaev" in Paris. The authors thank the organizers
for a stimulating conference and for excellent working conditions. 
They also  thank Charlie Frohman, Sam Gunningham,  Julien Korinman,  Thang Le, Julien March\'e  and George Pappas for  helpful discussions. 
During the course of this work  the first author was  partially supported by the projects AlMaRe (ANR-19-CE40-0001-01) and by the project ``CLICQ" of the R\'egion Bourgogne Franche Comt\'e.
The second author was partially supported by the NSF grants DMS-2004155 and DMS-2304033. The third author was partially supported by the Simons Foundation grant 957582.

 
 \section{Skein modules at roots of unity}
\label{sec:dimension}

In this section we study Kauffman bracket skein modules of closed 3-manifolds $M$ for roots of unity $A$ and we relate their dimensions
to the $SL(2,\C)$-character varieties of $M$, building upon the work of \cite{BW16} and \cite{FKL}. Then we use this relation to prove Theorem  \ref{t.main1-i} of the introduction.

\subsection{Definition of the Kauffman bracket skein module }
\label{sec:skein}
Given a commutative ring $R$ and an invertible element $A\in R$,
the Kauffman bracket skein module $S(M,R)$ of $M$ is the quotient of the free $R$-module spanned by isotopy classes of unoriented framed (a.k.a banded) links in $M$ (including $\emptyset$) modulo the Kauffman bracket skein relations:
\begin{figure}[h]
{\centering
\def \svgwidth{1.1\columnwidth}
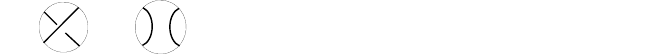}
\end{figure}

In this article, we will always make one of the following choices for the pair $(R,A):$ either $R=\Z[A^{\pm 1}]$, $R=\Q[A^{\pm 1}]$, $R=\Q(A)$ or $R=\C$ and $A=\zeta\in \C^*$. 
We will write $S(M, \Z[A^{\pm 1}])$, $S(M, \Q[A^{\pm 1}] )$, $S(M):=S(M,\Q(A))$ and $S_{\zeta}(M)$ for the corresponding types of the Kauffman bracket skein module. 
We will use an analogous notation, $S(\Sigma, \Z[A^{\pm 1}])$, $S(\Sigma, \Q[A^{\pm 1}] )$, $S(\Sigma)$ and $S_{\zeta}(\Sigma)$ for the skein algebra of a surface $\Sigma.$
Moreover, we will often abbreviate the Kauffman bracket skein module and algebra as the skein module and skein algebra, since we will not discuss any other skein modules nor algebras in this paper.

\subsection{Dimensions of skein modules at roots of unity.}
Important progress in our understanding  
of the skein modules  $S_{\zeta}(M)$ of closed $3$-manifolds $M$ at roots of unity $\zeta$, was made recently in the work of Bonahon and Wong \cite{BW16}, Frohman, Kania-Bartoszy\'nska and L\^e \cite{FKL19, FKL} and Ganev, Jordan, and Safronov \cite{GJS19a}. Building on these works, which reveal deeper connections between the skein modules and their  $\slC$-character varieties,  and using a result from \cite{Abelian} that relies on the theory of  the non-semisimple quantum 3-manifold invariants of Constantino, Geer and Patureau-Mirand \cite{NSINV} we prove the following important technical result of the paper:

\begin{theorem}\label{thm:inequality_dim}
If $M$ is a closed  $3$-manifold and $\zeta$ be a root of unity of order $2N$ with $N$ odd, then $\dim_{\C} (S_{\zeta}(M)) \geq |X(M)|$.
\end{theorem}

The proof is given in Subsection \ref{ss.proof-inequality_dim}. We precede it with a preliminary result in the next subsection.

\subsection{A construction of $\C[X(M)]$-equivariant maps on $S_{\zeta}(M)$}
\label{sec:equiv_maps}
Given a closed  $3$-manifold $M$ and a  $2N$-th  root of unity $\zeta$  with $N$ odd,
Bonahon and Wong \cite{BW16} (see also \cite{Le15}) showed that $S_{\zeta}(M)$ admits a natural structure of a module over $S_{-1}(M)$.  
To describe this structure, for $k\in \Z$ let $T_k(z) \in  \Z[z]$ be the $k$-th Chebyshev polynomial of the first kind, defined by
$$T_0(z)=2, \ \ \ T_1(z)=z, \ \ \  T_n(z)=z T_{n-1}(z)-T_{n-2}(z).$$
The polynomial $T_k$ is satisfying the identity $T_k(X+X^{-1})=X^k+X^{-k}$.

For a framed link $L$ in $M$, and $i\in \Z_{\geq 0}$, we write $L^i$ for the framed link in $M$ which consists of $i$ parallel copies of $L$ following the given framing.
By replacing $z^i$ with $L^i$, and and by linearity, we can consider $T_n(L)$ to be an element of $S_{\zeta}(M)$.
 It is proved in \cite{BW16} (see also \cite{Le15}) for a simplified proof) that for two framed links $L_0$ and $L_1$, the element $T_N(L_0)\cup L_1\in S_{\zeta}(M)$ depends on the homotopy class of $L_0$ only, and that  $T_N(L_0)$ satisfies the Kauffman relations for $A=-1$. 
Thus  $\C[\cX(M)]=S_{-1}(M)$ acts on $S_{\zeta}(M)$ by $T_N(L_0)\cdot L_1=T_N(L_0)\cup L_1$ and this action
provides a structure of $S_{-1}(M)$-module on $S_{\zeta}(M)$.

Furthermore, given a character $\chi\in X(M)$, the algebra  $\C[\cX(M)]=S_{-1}(M)$
also acts on $\C$ through
$$\C[\cX(M)]\to \C[X(M)], \ by \ f\cdot z=f(\chi)z,$$ for any $f\in \C[X(M)]$ and any $z\in \C$.

 \begin{theorem} \label{equivariant} Given a character $\chi\in X(M)$ that is the trace of a representation $\rho$, there is a surjective map
$RT_{\chi}: S_{\zeta}(M)\to \C$  that is $\C[\cX(M)]$-equivariant with respect to above two actions. 
\end{theorem}

Recall that an irreducible character is one that
is the trace of an irreducible $SL(2,\C)$-representation of $\pi_1(M)$,  and an abelian character is the trace of a diagonal $SL(2,\C)$-representation.
Finally, a central character is the trace of an $SL(2,\C)$-representation of $\pi_1(M)$ 
with values in the center $\{\pm I\}$ of $SL(2,\C)$.  Theorem \ref{equivariant}  follows from different techniques and results according to whether
$\rho$ is irreducible, central or non-central abelian.
We discuss each of these cases separately below.

\begin{remark}  \label{r.action} By \cite{Barrett}, any choice of spin structure on $M$ provides an isomorphism between the modules $S_A(M)$ and $S_{-A}(M)$ which yields an algebra isomorphism
$S_1(M)\simeq S_{-1}(M)$. Hence instead of assuming that $\zeta$ is a $2N$-th root of unity, as we do above, some authors (e.g  \cite{GJS19a} ) choose $\zeta$ to be a a primitive $N$-th root of unity. Note that in this case the argument described before the statement of Theorem \ref{equivariant} 
will give a structure of $S_1(M)$-module on $S_{\zeta}(M)$. 
\end{remark} 
\vspace{0.03in}

\noindent{\bf{Irreducible characters:}} Let $\rho: \pi_1(M) \longrightarrow SL(2,\C)$ be an irreducible representation and  $\zeta$ a root of unity.
The character of  $\rho$ in $\cX(M)$ determines  a maximal ideal $\mathfrak{m}_{\rho}$ of the ring $S_{-1}(M)\simeq \C[\cX(M)]$, which in turn acts on
$S_{\zeta,\rho}(M)$, in the sense of Remark \ref{r.action}, giving a submodule $\mathfrak{m}_{\rho}S_{\zeta}(M)$ of $S_{\zeta}(M)$.
Frohman, Kania-Bartoszy\'nska and L\^e \cite{FKL}
defined the reduced skein module of $M$ at $\rho$ by
$$S_{\zeta,\rho}(M):=S_{\zeta}(M)/\mathfrak{m}_{\rho}S_{\zeta}(M).$$

Recall that the skein module $S_{-1}(M)$ has a natural structure 
 of $\C$-algebra isomorphic with $\C[\cX(M)]$. The isomorphism $\psi: S_{-1}(M)\to \C[\cX(M)]$ maps any framed link $L=L_1 \cup \ldots \cup L_k$ in $M$ to $(-1)^{k} t_L$ where $t_L=t_{L_1}\cdot \ldots \cdot t_{L_k}$, where $t_{L_i}$ is the trace function of $L_i$ with its framing ignored, \cite{PS00}.

Therefore,
$I_\rho$ is generated by skeins $T_N(L) - (-1)^k t_L(\rho)$ for links $L\in S_\zeta(M).$
Note that this notion of ``reducedness'' is unrelated to that of algebraic varieties also used in this paper.

\begin{theorem}\rm{(\cite[Thm. 7]{FKL}\label{thm:reduced_skein}) }If $\zeta$ is a root of unity of order $2N$ with $N$ odd, then for any irreducible representation $\rho$, 
$S_{\zeta,\rho}(M)$ is a non-trivial vector space over $\C$.
\end{theorem}

In fact, by \cite{FKL}, $S_{\zeta,\rho}(M) \simeq \C$, but for our purposes the non-triviality of $S_{\zeta,\rho}(M)$ suffices. Furthermore, we only need a small part of the sophisticated machinery of \cite{FKL} only. A key ingredient in the proof of the \cite{FKL} result is the fact that the Azumaya locus of $S_{\zeta}(\Sigma)$ contains all the irreducible characters of $X(\Sigma)$, for surfaces which are not necessarily closed. This fact, for closed $\Sigma$, as in our case can be also deduced from  \cite[Thm 1.2]{GJS19a}. With this result at hand, we can deduce Theorem \ref{thm:reduced_skein} by the argument of Section 12 of \cite{FKL}.

For each irreducible character $\chi,$ let us realize it as the trace of an irreducible representation $\rho$ and, using Theorem \ref{thm:reduced_skein}, let us fix an epimorphism $S_{\zeta,\rho}(M) \to\C$ and denote the projection $S_{\zeta}(M)\rightarrow S_{\zeta,\rho}(M)\to \C$ by $RT_{\chi}$.
Recall that for any $\chi\in X(M),$ the algebra $\C[\cX(M)]$ acts on $S_{\zeta}(M)$ and on $\C$. By the definitions of these actions it is easy to see that the map $RT_{\chi}: S_{\zeta}(M)\to \C$ is $\C[\cX(M)]$-equivariant with respect to these actions. Hence Theorem \ref{equivariant} follows for irreducible characters.
\vspace{0.05in}

\noindent{\bf{Central characters:}}  Now we  explain how to associate $\C[\cX(M)]$-equivariant maps $S_{\zeta}(M)\to \C$ with central characters. Here, we can not appeal to the arguments of \cite{FKL}, since central characters are not in the Azumaya locus of skein algebras at roots of unity. One can however construct such $\C[\cX(M)]$-equivariant maps using Reshetikhin-Turaev invariants \cite{RT91}. We will describe this construction by utilizing the skein-theoretic approach of \cite{BHMV}.

For $i\geq 0$, let $[i]=\frac{\zeta^{2i}-\zeta^{-2i}}{\zeta^2-\zeta^{-2}}$ and let $S_i(z)\in\Z[z]$ be the Chebyshev polynomials of the second kind: $S_0(z)=1, S_1(z)=z$ and $S_{i+1}(z)=zS_i(z)-S_{i-1}(z)$. 
We will denote by $\langle L \rangle$ the Kauffman bracket of a framed link $L$ in $S^3$. For a framed link $L \subset S^3$ whose components $L_1,\ldots,L_n$ are colored by polynomials $f_1,\ldots ,f_n\in \Z[A^{\pm 1}, z]$, the Kauffman multi-bracket $\langle L, f_1,\ldots,f_n\rangle$ is defined as follows: if all $f_i$'s are monomials $z^{n_i}$ then this is just the Kauffman bracket of the framed link consisting of $n_i$ parallel copies of $L_i$ for each $i$. These parallel copies have linking numbers zero between themselves.   We extend this definition multilinearily to all polynomials $f_1,...,f_n$. We will simply write $(L,f)$ if all components of $L$ are colored by the same color $f$.

Let $M$ be a closed $3$-manifold obtained by a $0$-surgery on a framed link $L_M$ in $S^3$.
 For another link $L$ in $M$, the $\mathrm{SO}_3$-Reshetikhin-Turaev invariant of the pair $(M,L)$ at the $2N$-root of unity $\zeta$ is defined by
$$RT^{\zeta}(M,L)=\frac{\langle (L_M,\omega_N)\cup L \rangle}{\langle U_+\rangle^{n_+} \langle U_- \rangle^{n_-}},$$
where the components of $L_M$ are colored by the Kirby color $\omega_N=\underset{i=0}{\overset{\frac{N-3}{2}}{\sum}} (-1)^i [i+1] S_i(z)$,  the link $U_+$ (resp. $U_-$) is the $+1$  (resp. $-1$) framed unknot, and the signature of the linking matrix of $L$ is $(n_+,n_-)$. By \cite{BHMV}, the above is a topological invariant of the pair $(M,L)$. It is clear from the definition that the map $L\in M \rightarrow RT^{\zeta}(M,K)$ satisfies the Kauffman skein relations for $A=\zeta$. Hence it induces a map
$$\begin{array}{rccl}RT_{0}: &  S_{\zeta}(M) &\longrightarrow &\C
\\      & L & \longrightarrow & RT^{\zeta}(M,L).
\end{array}$$
Now for any $c \in H^1(M,\Z/2\Z)$, we can also define a map 
$$\begin{array}{rccl}RT_{c}: &  S_{\zeta}(M) &\longrightarrow &\C
\\      & L & \longrightarrow & (-1)^{c(L)}RT^{\zeta}(M,L).
\end{array}$$
Note that since the center of $\slC$ is $\lbrace \pm I_2 \rbrace$, central characters in $X(M)$ can be identified with cohomology classes in $H^1(M,\Z/2\Z)$. 

\begin{proposition}\label{prop:RT-equivariant} For any $c\in H^1(M,\Z/2\Z)$, the map $RT_c: S_{\zeta}(M) \rightarrow \C$, is a surjective $\C[\cX(M)]$-equivariant map, where the $\C[\cX(M)]$-action on $\C$ corresponds to the cohomology class $c$.
\end{proposition}
\begin{proof}
Note that since for any link $L$ we have $RT_c(L)=\pm RT_0(L)$, it will be sufficient to prove that $RT_0$ is surjective. Moreover, since $N$ is odd, the $N$-th Chebyshev polynomial $T_N$  is an odd polynomial. Hence,
$$RT_c(T_N(L_0) \cup L_1)=(-1)^{c(L_0)+c(L_1)} RT_0(T_N(L_0) \cup L_1),$$ 
for any links $L_0$ and $L_1$ in $M$. Therefore, it will also be sufficient to prove that $RT_0$ is $\C[\cX(M)]$-equivariant, for the $\C[\cX(M)]$-action on $\C$ given by $f\cdot z=f(\rho_0)z$, where $\rho_0$ is the trivial representation.

Let us show that $RT_0$ is surjective. Note that if $K$ is a knot in $S^3$ colored by $S_i$, and if $m_K$ is the $0$-framing trivial knot which is a meridian of $K$, then, by for example \cite[Section 3]{BHMV},
$$\langle (K,S_i) \cup m_K \rangle= (-\zeta^{2i+2}-\zeta^{-2i-2})\langle (K,S_i) \rangle.$$ 
We note that $-\zeta^{2i+2}-\zeta^{-2i-2}\ne -\zeta^2-\zeta^{-2}$ for $0<i \leq \frac{N-3}{2}$  since $\zeta$ is a primitive $2N$-th root of unity. Let $Q(z)$ be a polynomial such that $Q(-\zeta^{2i+2}-\zeta^{-2i-2})=0$ for all $1\leq i \leq \frac{N-3}{2}$, and such that $Q(-\zeta^2-\zeta^{-2})=1$. 
Now let $m_L$ be the $0$-framing unlink consisting of one meridian of each component of $L$ and $L_M$ be the $0$-surgery presentation for $M$ as before.
Then 
$$\langle (L_M,\omega) \cup (m_L,Q(z)) \rangle = \langle (L_M,S_0)\rangle = 1,$$
 which shows that $RT_0((m_L,Q(z)))\neq 0$.

Finally, let us show that $RT_0$ is $\C[\cX(M)]$-equivariant. We want to show that
$$RT_0(T_N(L_0) \cup L_1)=(-2)^{\sharp L_0} RT_0(L_1),$$ for any framed links $L_0$ and $L_1$ in $M$, where $\sharp L_0$ is the number of components of $L_0$. Then 
$\langle (L_M,\omega) \cup (L_0,T_N(z)) \cup (L_1,z) \rangle$ is preserved by homotopies of $L_0$. Since we are considering $L_M,L_0$ and $L_1$ as links in $S^3$, it is sufficient to look at the case where $L_0$ is an unlink disjoint from $L_M\cup L_1$. Then, 
$$RT_0(T_N(L_0)\cup L_1)= RT_0(L_1)\left(T_N(-\zeta^2-\zeta^{-2})\right)^{\sharp L_0}= (-2)^{\sharp L_0} RT_0(L_1),$$ as $\zeta^{2N}=1$ and $N$ is odd.
\end{proof}

\vskip 0.03in

\noindent{\bf{Non-central abelian characters:}} The relation of non-central abelian characters in $X(\Sigma)$
to  the Azumaya locus of $S_{\zeta}(\Sigma)$ was studied in \cite{KK22} however we do not utilize this study here. Instead we take a more direct approach, using a  result proved by the first author of this paper in \cite{Abelian}, and which is the analogue of 
 Proposition \ref{prop:RT-equivariant}  for non-central abelian characters.
 
A class $\omega\in H^1(M,\C/\Z)$ defines $\pi_1(M)\to H_1(M)\to \C^*=\C/\Z$. Every abelian $SL(2,\C)$-representation arises in this way.

\begin{theorem}\rm{(\cite{Abelian})}\label{RT-equivariant1} For any $\omega \in H^1(M,\C/{\Z})$, there is a map $RT_{\omega}: S_{\zeta}(M) \rightarrow \C$ that  is a surjective $\C[\cX(M)]$-equivariant map, where the $\C[\cX(M)]$-action on $\C$ corresponds to the cohomology class $\omega$.
\end{theorem}

Theorem \ref{equivariant} follows at once from Theorem \ref{RT-equivariant1}, in this case.
Given a  homology class $\omega \in H^1(M,\C/{\Z})$  \cite{NSINV} constructs invariants of closed 3-manifolds based on representation theory 
of the unrolled $sl_2$ quantum group. The proof of Theorem \ref{RT-equivariant1} uses these non-semisimple versions of 3-manifold invariants
 in a way analogous as we did with
the  skein theoretic Reshetikhin-Turaev invariants  of \cite{BHMV} in Proposition  \ref{prop:RT-equivariant} .
The proof, however,
is more involved than this of  Proposition  \ref{prop:RT-equivariant} also partly because \cite{NSINV} 
works with $2N$-th roots of unity, where $N$ is even, and the theory  in the case  where $N$ is odd has some subtle, but important for our purposes differences,   that are  worked out  in
detail in \cite[Theorem 1.2]{Abelian}.

\vskip 0.04in

\begin{remark} Following Bonahon-Wong and L\^e, \cite{BW16, Le15} a finite dimensional representation
$f: S_{\zeta}(M) \rightarrow  {\rm End}(V)$ is said to have \underline{classical shadow} $\chi \in X(M)$,  
if for any framed knot $K$ and any framed link $L$
we have
$$f(L\cup T_N(K))=-\chi(K) f(L).$$

From this view point, given a non-central abelian character $\chi \in X(M)$,   Theorem \ref{RT-equivariant1} (and its proof) uses the theory  of
\cite{NSINV}  to construct an explicit finite diagonal representation of $S_{\zeta}(M)$ with classical shadow $\chi$.

\end{remark}

\subsection{Proof of Theorem \ref{thm:inequality_dim} } 
\label{ss.proof-inequality_dim}
Let us fix a root of unity $\zeta$ of order $2N$ with $N$ odd, and a closed 
$3$-manifold $M$. Let $X=\lbrace \chi_1,\ldots,\chi_n \rbrace$ be either $X(M)$ if $|X(M)|<\infty$, or a finite subset of $X(M)$ otherwise. We define a $\C[X(M)]$-equivariant map
$$\begin{array}{rccl} RT: & S_{\zeta}(M) &\longrightarrow & \C^X\\ 
 & L & \longrightarrow & (RT_{\chi_1}(L),\ldots,RT_{\chi_n}(L)),
\end{array}$$
 where $RT_{\chi_i}$ are the maps defined in Section \ref{sec:equiv_maps} and $\C[X(M)]$ acts on $\C^X$ by 
 $$f\cdot(x_1,\ldots,x_n)=(f(\rho_1)x_1,\ldots,f(\rho_n)x_n).$$

Theorem \ref{thm:inequality_dim} follows immediately from the next lemma.
\vskip 0.04in

\begin{lemma} The map $RT$ is surjective onto $\C^X$. Hence, $\dim_{\C} (S_{\zeta}(M)) \geq dim_{\C}\, \C^X=|X|$.
\end{lemma}

\begin{proof}
Since the maps $RT_{\chi_i}$ are surjective onto $\C$, there are links $L_1,\ldots, L_n$ such that  $RT_{\chi_i}(L_i)\neq 0$, for each $i=1,...,n$.  By an easy induction on $k$, for any $1\leq k \leq n$, we can choose scalars $t_1,\ldots,t_k \in \C$ such that $RT_{\chi_i}(t_1K_1+\ldots +t_k K_k) \neq 0$ for each $1\leq i \leq k$.
Indeed, given $t_1,\ldots,t_k$, $RT_{\chi_i}(t_1K_1+\ldots +t_kK_k +t K_{k+1})$ can vanish
for finitely many values of $t\in \C$ only, since $RT_{\chi_{k+1}}(K_{k+1}) \neq 0$.

Now take $t_1,\ldots,t_n$ such that $RT_{\chi_i}(t_1K_1+\ldots +t_n K_n) \neq 0$ for all $1\leq i \leq n$, and write $x=t_1 K_1+\ldots +t_nK_n$. Let us pick functions $f_i\in \C[X(M)]$ such that $f_i(\rho_j)=\delta_{i,j}$. Then by the $\C[X(M)]$-equivariance of $RT$, the elements $f_i\cdot x$ are mapped to non zero-scalar multiples of the canonical basis of $\C^X$. Therefore the map $RT$ is surjective.
\end{proof}


\section{On dimensions and bases of skein modules}
\label{sec:main-thm-proof}

This section contains the proofs of our main general results about skein modules. In Sections \ref{sec:reduceness_41}-\ref{ss.torusknot} we will apply thee results
to study skein modules of 3-manifolds obtained
surgery on the figure-eight and the $(2,2n+1)$-torus knots.

We define the \underline{$(A+1)$-torsion} of $S(M)$ to be
$$S^{A+1}(M, \Q[A^{\pm 1}])):=\{x\in S(M, \Q[A^{\pm 1}])): (A+1)^nx=0 \ \text{for some}\ n\in \Z_{>0}\}.$$

Recall that a $\Q[A^{\pm 1}]$-module $S$ is \underline{tame} if it is a direct sum of cyclic $\Q[A^{\pm 1}]$-modules and $S$ does not contain $\Q[A^{\pm 1}]/(\phi_{2N})$ as a submodule, for at least one odd $N$, where $\phi_{2N}$ is the $2N$-th cyclotomic polynomial. 

In this section, we use Theorem \ref{thm:inequality_dim} to prove Theorem  \ref{t.main1-i}  of the Introduction, 
which we restate here in a slightly different form:

\begin{theorem}\label{t.main} If $M$ is a closed
$3$-manifold with tame $S(M,\Q[A^{\pm 1}]))$ then 

\begin{enumerate}[(a)]
\item  $|X(M)|\leq \dim_{\Q(A)} S(M)\leq \dim_{\C}\, \C[\cX(M)]$.
\item $\dim_\Q\, S^{A+1}(M,\Q[A^{\pm 1}])\leq \dim_\C\, \sqrt{0}$.

 \end{enumerate}
 \noindent In particular, when $\cX(M)$ is reduced, then $\dim_{\Q(A)} S(M)=|X(M)|$ and $S(M,\Q[A^{\pm 1}]))$ has no $(A+1)$-torsion.

 \end{theorem}
 
\begin{proof} 
Since $S(M, \Q[A^{\pm 1}])$ is tame, by definition, we have
$$S(M, \Q[A^{\pm 1}])=F\oplus ( \underset{i}{\bigoplus}\Q[A^{\pm 1}]/(q_i^{s_i})),$$
where $F$ is a free $\Q[A^{\pm 1}]$-module and the sum is over certain powers of certain monic irreducible polynomials $q_i\in \Q[A]$, $q_i\ne 1$, possibly repeating themselves.

 For any primitive $2N$-th root of unity $\zeta$, and ${q_i}^{s_i}$ as above, either $q_i$ is  the $2N$-cyclotomic polynomial $\phi_{2N}$, or 
$$\Q[A^{\pm 1}]/({q_i}^s) \underset{A=\zeta}{\otimes} \C =0.$$

By our definition of tameness, there is an odd $N_1>0$ such that
$\Q[A^{\pm 1}]/(\phi_{2{N_1}}^s)$ is not a summand of $S(M, \Q[A^{\pm 1}])$
for any power $s$. Hence, for any primitive $2{N_1}$-th root $\zeta$ of unity,
$$S_{\zeta}(M)=S(M, \Q[A^{\pm 1}]) \underset{A=\zeta}{\otimes} \C \simeq \C^{\dim F}.$$  
Thus, by Theorem \ref{thm:inequality_dim},  we have
$$\dim S(M) =\dim F\geq |X(M)|,$$ 
proving the lower inequality of part (a) in the statement of the theorem.
We have
\begin{equation}\label{e.dimeq1}
\dim_\C\, \C[\cX(M)]=
|X(M)|+\dim_\C\, \sqrt{0}.
\end{equation}
Since $S(M, \Q[A^{\pm 1}])\underset{A=-1}{\otimes}\C = S_{-1}(M)\simeq \C[\cX(M)]$, we also have

\begin{equation}\label{e.dimeq}
\dim F+\dim_\Q\, S^{A+1}(M,\Q[A^{\pm 1}])=\dim_\C\, \C[\cX(M)]
\end{equation}
Here, we use here the fact that $\Q[A^{\pm 1}]/(q^s)\underset{A=-1}{\otimes}\C =0$ for a monic irreducible $q\in \Q[A]$ unless $q=A+1$.
Since the left side of \eqref{e.dimeq} is 
$$\dim_{\Q(A)} S(M)+\dim_\Q \, S^{A+1}(M, \Q[A^{\pm 1}]),$$
the upper inequality of part (a) follows at once and part (b) follows from (a) and  \eqref{e.dimeq1}.

Finally, if in addition  $\cX(M)$ is reduced then $\dim_\C\, \C[\cX(M)]=|X(M)|$ and the claim follows from parts (a) and (b).
\end{proof}

Theorem \ref{t.main} shows that if  $X(M)$  is infinite, then the skein module $S(M, \Q[A^{\pm 1}])$ is not tame. The next proposition provides more information about the structure of the skein module
of 3-manifolds in this case.

\begin{proposition}\label{prop:infiniteX}
For any closed 3-manifold $M$ with infinite $X(M)$,  either $S(M, \Q[A^{\pm 1}])$ is  not a sum of cyclic $\Q[A^{\pm 1}]$-modules  or it contains a submodule  $(\Q[A^{\pm 1}]/(\phi_{2N}))^\infty$ for every odd $N$.
\end{proposition}

Note that the first possibility holds for $S(\R\P^3\# \R\P^3)$, \cite{Mro11}, while the second holds for $S(S^1\times S^2)$, \cite{HP95}.

\begin{proof} Assume, for a contradiction,  that $S(M,\Q[A^{\pm 1}])$ is a sum of cyclic $\Q[A^{\pm 1}]$-modules and it does not contains a submodule  $(\Q[A^{\pm 1}]/(\phi_{2N}))^\infty$ for some odd $N$.
Then it contains a finite number $C$ summands of the form $\Q[A^{\pm 1}]/(\phi_{2N}^s)$, for $s\in \Z$, only.

Let $F$ be the free part of $S(M,\Q[A^{\pm 1}])$ as before. Then, by \cite{GJS19},
$\dim_{\Q[A^{\pm 1}]} F=\dim_{\Q(A)} S(M)$ is finite and 
$$|X(M)|\leq \dim S_\zeta(M)=\dim F+C<\infty,$$
contradicting the hypothesis that  $X(M)$ is infinite.
\end{proof}

The next proposition shows, in particular that when $S(M,\Q[A^{\pm 1}])$ is tame and $\cX(M)$ is reduced the problem of finding a basis of the skein module $S(M)$
is reduced to finding a basis for the coordinate ring $\C[X(M)]$.

Recall that for a $3$-manifold $M$ and $\gamma\in \pi_1(M),$ the  \underline{trace function} $t_\gamma: X(M)\to \C$ is defined by $t_\gamma([\rho])=Tr \rho(\gamma)$. 
More generally, we define the trace function $t_c$ for any unoriented loop $c\subset M$ as $t_\gamma$ for any $\gamma\in \pi_1(M)$ in the free homotopy class of $c$. Since $t_\gamma$ is preserved by conjugation and inversion of $\gamma$, this $t_c$ is well defined.
For a simple closed curve in $\p E_K$ representing  the slope $p/q\in \Q\cup\{1/0\}$, 
we will use $t_{p/q}$ to denote its trace function. In particular, $t_{1/0}$ denotes the trace function of the meridian of $K$.

Recall that the skein module $S_{-1}(M)$ has a natural structure 
 of $\C$-algebra isomorphic with $\C[\cX(M)]$. The isomorphism $\psi: S_{-1}(M)\to \C[\cX(M)]$ maps any framed link $L=L_1 \cup \ldots \cup L_k$ in $M$ to $(-1)^{k} t_L$ where $t_L=t_{L_1}\cdot \ldots \cdot t_{L_k}$, where $t_{L_i}$ is the trace function of $L_i$ with its framing ignored, \cite{PS00}.

\begin{proposition}\label{p.basis} 
(a) Let $M$ be a closed $3$-manifold  such that $S(M,\Q[A^{\pm 1}])$ has no $(A+1)$-torsion, and 
 let $\{b_i\}_{i\in I}$ be a subset of $S(M,\Q[A^{\pm 1}])$ whose images 
under the natural map $$S(M, \C[A^\pmo])\to S(M, \C[A^\pmo])\otimes_{\C[A^\pmo]} \C\simeq  S_{-1}(M)$$
are linearly independent in $S_{-1}(M)\simeq \C[\cX(M)]$. Then the natural map
$$S(M,\Q[A^{\pm 1}])\to S(M,\Q[A^{\pm 1}]) \otimes_{\Q[A^\pmo]} \Q(A)\simeq S(M)$$ maps $\{b_i\}_{i\in I}$ to a linearly independent family in $S(M)$.
\vskip 0,03in
\noindent (b) If, moreover,  $S(M,\Q[A^{\pm 1}])$ is tame and the image of $\{b_i\}_{i\in I}$ in $S_{-1}(M)$ is a basis then
$\{b_i\}_{i\in I}$ maps to a basis of $S(M)$.
\end{proposition}

\bpr  Let $\mathcal{B}=\lbrace b_i\rbrace_{i\in I}$ be  a finite collection of elements of $S(M, \Q[A^{\pm 1}])$ whose images in $S_{-1}(M)$ form a basis of $S_{-1}(M).$ 
We claim that $\mathcal{B}$ must be $\Q[A^{\pm 1}]$-linearly independent in $S(M)$. Indeed, assume that for some not all zero polynomials $P_1,.., P_n\in \Q[A^{\pm1}]$ one has 
 $$P_1(A)b_1+ \ldots +P_n(A)b_n=0.$$ Since there is no $(A+1)$ torsion in $S(M, \Q[A^{\pm 1}])$,
by dividing $P_1,..., P_n$ by a power of $A+1$ if necessary, one can assume that $P_i(-1)\neq 0$ for some $i$. Then, by evaluating at $A=-1$, this contradicts the fact that the $b_1,..,b_n$ are linearly independent in $S_{-1}(M,\Q[A^{\pm 1}])$. This proves part (a).
 
To see (b) assume that $S(M, \Q[A^{\pm 1}])$ is tame. By Theorem \ref{t.main},
$$\dim_{\C} S_{-1}(M)\geq\dim_{\Q(A)}S(M).$$ 
If the $b_i$ are chosen to be a basis of $S_{-1}(M),$ since their images are linearly independent, we actually have $\dim_{\C}S_{-1}(M)=\dim_{\Q(A)}S(M)$ and their images are a basis of $S(M)$.
\epr

 As a corollary of Proposition \ref{p.basis}, if $S(M,\Q[A^{\pm 1}])$ has no $(A+1)$-torsion and $$\dim_{\C} S_{-1}(M)=\dim_{\Q(A)}S(M)$$ (for example, when $S(M, \Q[A^{\pm 1}])$ is tame and $\cX(M)$ is reduced) and the image of $\{b_i\}_{i\in I}$ in $S_{-1}(M)$ is its basis then
$\{b_i\}_{i\in I}$ maps to a basis of $S(M)$.

\begin{remark} Invariant theory provides an algorithm for a finding finite presentation of the commutative algebra $\C[\cX(M)],$ based on the finite presentation of $\pi_1(M),$ see \cite{BH95}. If $\C[\cX(M)]=\C[x_1,...,x_k]/I$, then for any Gr\"obner basis of $I$, the set of standard monomials with respect to that basis forms a $\C$-basis of
$\C[\cX(M)],$ see \cite{Grobner}. Consequently, Proposition \ref{p.basis} enables an algorithmic construction of bases of skein modules $S(M)$ with $S(M, \Q[A^{\pm 1}])$ tame and $\cX(M)$ is reduced.
\end{remark}

\section{Examples with tame skein modules}
\label{sec:examples}

In this section, we give examples of closed 3-manifolds $M$ with $S(M, \Z[A^{\pm 1}])$ finitely generated  over $\Z[A^{\pm 1}],$ hence, tame $S(M, \Q[A^{\pm 1}])$. We rely on results of the first author,
\cite{Det21}, on skein modules of 3-manifolds obtained by surgery on knots. Moreover, we will use computations of the action of the skein algebra of the boundary torus on the skein module of the figure eight-knot and the $(2,2n+1)$-torus knots by Gelca and Sain \cite{GS04, GS03}, cf. Theorem
\ref{t.fin_gen_skein} below.

For any $3$-manifold with boundary $M$, its skein module $S(M, \Z[A^{\pm 1}])$ has a natural left module structure over the skein algebra $S(\partial M, \Z[A^{\pm 1}])$ of the boundary, induced by the homeomorphism 
$$M \sqcup_{\p M\simeq \p M \times \{0\} } \p M \times [0,1] \simeq M.$$ Recall that for a knot $K$, the complement of an open tubular neighborhood of $K$ in $S^3$ is denoted by $E_K.$ Its skein module $S(E_K, \Z[A^{\pm 1}])$ has thus a natural left 
$S(T^2)$-module structure, where $T^2=\p E_K$ is the $2$-torus with distinguished meridian $m$ and longitude $l$. The orientation of $S^3$ induces an orientation of $m$ and $l$ up to simultaneous reversal.

The skein algebra of the torus has a nice description, originally discovered by Frohman and Gelca \cite{FG00}. 
For coprime integers $p,q$, let $(p,q)_T$ denote the simple closed curve on $T^2$ of slope $p/q$ (representing $\pm (p m +ql)$ in $H_1(\partial E_K))$. Additionally, we set $(0,0)_T$, to be the empty multicurve, $\emptyset$. More generally, for any $p,q\in \Z$ we set 
$$(p,q)_T=T_d((p/d,q/d)),$$ where $d=gcd(p,q)$ and $T_d(X)$ is the $d$-th Chebyshev polynomial of the first kind, as before. We warn the reader that with our conventions for the slopes, $(p,q)_T$ corresponds to $(-q,p)_T$ in \cite{GS03} or \cite{GS04}.
Let $\Z^2/_{\lbrace \pm 1\rbrace}$ denote the quotient of $\Z^2$ by the involution $(p,q)\to (-p,-q).$
Since multicurves on $T^2$ without contractible components form a basis of $S(T^2)$, the  skeins $(p,q)_T$ 
for $\pm (p,q) \in\Z^2/_{\lbrace \pm 1\rbrace}$ form a basis of $S(T^2)$.

Let $\mathcal{T}=\Z[A^{\pm 1}]\langle \mu,\lambda\rangle/(\mu\lambda - A^2\lo\mu)$ be the quantum torus
and let $\theta:\mathcal{T}\rightarrow \mathcal{T}$ be the algebra homomorphism defined by $\theta(\mu^k \lambda^l)=\mu^{-k}\lambda^{-l}$. Let $\mathcal{T}^{\theta}$ be the $\theta$-invariant subalgebra of $\mathcal{T}$. Finally, for any $p,q \in \Z^2$, let $e_{p,q}=A^{-pq}\mu^p\lambda^q$.

\begin{theorem}\label{thm:isom_skein_alg}\cite{FG00,Sa} The map $S(T)\to \mathcal{T}^{\theta}$ sending
$$(p,q)_T \longrightarrow e_{p,q}+e_{-p,-q},$$
for $(p,q)\in \Z^2\setminus\lbrace (0,0)\rbrace,$ and $(0,0)_T$ to $e_{0,0},$
is an isomorphism of algebras.
\end{theorem} 

Recall that for $E_K$ a knot complement, $E_K(p/q)$ denotes the Dehn $p/q$-filling of $E_K.$ We will always assume that $p$ and $q$ are coprime.
In \cite[Corollary 5.3]{Det21}, the first author gave a criterion for Dehn fillings $E_K(p/q)$ of a knot $K$ to have $S(E_K(p/q))$ finitely generated over $\Z[A^{\pm 1}].$ We will need a slight generalization of this result. 
By Theorem \ref{thm:isom_skein_alg}, any non-zero $z\in S(\partial E_K)$ can be viewed  as an element of $\mathcal{T}^{\theta}\subset \mathcal{T}$, and thus can be written as $z=\sum_{(p,q)\in \Z^2} z_{(p,q)}e_{p,q},$ with the further condition that $z_{(p,q)}=z_{(-p,-q)}.$
The Newton polygon of $z$ is the convex hull of $\{(p,q)\in \Z^2: z_{(p,q)}\ne 0\}.$ 
A \underline{slope} of a polygon $P$ with vertices in $\Z^2$ is the slope of a side of $P$. The following theorem gives a criterion for Dehn fillings of $E_K$ to have finitely generated skein modules over $\Z[A^{\pm 1}]$ or $\Q[A^{\pm 1}]$. In particular, these skein modules are tame.

\begin{theorem}\label{thm:Dehn-filling_fin_gen}\cite{Det21} Let  $R$ be any commutative ring with a distinguished invertible $A\in R$. Assume that $E_K$ is a $3$-manifold with $\partial E_K=T^2,$ and that
$S(E_K, R)$  is finitely generated over $S(T^2,R)$, with generators $f_1,\ldots ,f_n.$ Suppose, moreover, that for each $1\leq i \leq n,$ there is an element $z_i \in S(T^2,R),$  such that 
\begin{enumerate}[(a)]
	\item The coefficients corresponding to the corners of the Newton polygon $P_i$ of $z_i$ are invertible in $R$.
	\item $z_i\cdot f_i \in \mathrm{Span}_{S(T^2,R)}\lbrace f_1,\ldots,f_{i-1}\rbrace.$
\end{enumerate} 
Then for any slope $p/q$ which is not a slope of any of the polygons $P_i$, the  skein module $S(E_K(p/q), R) $ is finitely generated over $R$.
\end{theorem}

\begin{proof}
	Note that $S(E_K(p/q),R)$ is generated by links in $E_K,$ since any link can be isotoped away from the Dehn-filling solid torus.
	We will prove by induction on $1\leq k \leq n$ that $S(T^2,R)f_1 + \ldots + S(T^2,R)f_k$, viewed as a subspace of $S(E_K(p/q),R)$, is finitely generated  over $R.$
	
	For $k=1$ we have $z_1\cdot f_1=0$ in $S(E_K,R)$ and the argument in \cite[Corollary 5.3]{Det21} can be applied without change to show that $S(T^2,R)f_1 \subset S(E_K(p/q),R)$ is finitely generated over $R,$ for any slope $p/q$ that is not a slope of $P_1.$
	
	To prove the inductive step, notice that it is sufficient to show that the quotient $$S(T^2,R)f_{k+1}/  \mathrm{Span}_{S(T^2,R)}\lbrace f_1,\ldots,f_{k}\rbrace$$ is finitely generated over $R.$ Let $\equiv$ denote the relation "being equal up to an element of $\mathrm{Span}_{S(T^2,R)}\lbrace f_1,\ldots,f_{k}\rbrace$". We have that $z_{k+1}\cdot f_{k+1}\equiv 0.$ The argument of \cite[Corollary 5.3]{Det21} can be reproduced, replacing equalities by $\equiv$ to show that $$S(T^2,R)f_{k+1}/  \mathrm{Span}_{S(T^2,R)}\lbrace f_1,\ldots,f_{k}\rbrace,$$ is finitely generated over $R$ for any slope $p/q$ that is not a slope of $P_{k+1}.$
\end{proof}

As an application of Theorem \ref{thm:Dehn-filling_fin_gen} we will get the following:

\begin{theorem}\label{t.fin_gen_skein} The skein module $S(E_K(p/q), \Z[A^{\pm 1}])$ is finitely generated over $\Z[A^{\pm 1}]$, if
\begin{enumerate}[(a)]
\item the knot $K$ is  the figure-eight and $p/q\notin \lbrace 0,\pm 4 \rbrace$; or
\vskip 0.02in
\item the  knot $K$ is the $(2,2n+1)$-torus knot for some $n\in \Z$ and $p/q\notin \lbrace 0, 4n+2 \rbrace$.
\end{enumerate}
\end{theorem}

\begin{proof}  First, we note that skein modules of two-bridge knot complements, which include both the 
figure-eight knot and $(2,2n+1)$-torus knots, are finitely generated over the skein algebra of their boundary, \cite[Theorem 2]{Le06}. For the figure eight knot, it is noted in \cite{GS04} that the generators can be taken to be the empty link $\emptyset,$ and two other links, denoted by $Y$ and $Z$ in \cite{GS04}.

 Gelca and Sain have computed the peripheral ideal of the figure-eight knot in \cite{GS04}. In particular, Proposition 5.4 of \cite{GS04} states that the element
\begin{eqnarray*}z_1=A^{-6}(3,-2)_T-A^6(-1,-2)_T+A^3(7,-1)_T-A(5,-1)_T
\\+(-A^{11}+A^3-A^{-1}+A^{-9})(3,-1)_T+(A^9-A^5-A^{-7})(1,-1)_T
\\+(-A^{11}+2A^7+A^3-A^{-1}+A^{-9})(-1,-1)_T+(A^{13}+A)(-3,-1)_T
\\-A^{-1}(-5,-1)_T+A^8(7,0)_T+(-2A^8+A^4-A^{-4})(5,0)_T
\\+(-A^{12}+A^8-A^4-1+A^{-4})(3,0)_T+(A^{12}-A^8+1+A^{-4})(1,0)_T 
\end{eqnarray*}
is in the peripheral ideal  of the figure-eight knot complement, that is $z_1\cdot \emptyset=0.$ (Again we stress that we use different conventions for the slopes of curves on $\partial E_K,$ with $(p,q)_T$ here corresponding to $(-q,p)_T$ in \cite{GS03} or \cite{GS04}).
One can readily see that the Newton polygon associated to this element is the convex hull of the points 
$$(7,0),(7,-1),(3,-2),(-1,-2),(-5,-1),(-7,0),(-7,1),(-3,2),(1,2) \ \text{and}\ (5,1),$$ 
and that the corner coefficients are all of the form $\pm A^k$. The slopes of this polygon are $0,-2,4,-4$ and $\infty$.
Moreover, Proposition 5.4 of \cite{GS04} also states that the element
\begin{eqnarray*}A^6(-3,-2)_T-A^{-6}(1,-2)_T+A^{-3}(-7,-1)_T-A^{-1}(-5,-1)_T
\\ +(-A^{-11}+A^{-3}-A+A^9)(-3,-1)_T+(A^{-9}-A^{-5}-A^7)(-1,-1)_T
\\ +(-A^{-11}+2A^{-7}+A^{-3}-A+A^9)(1,-1)_T+(A^{-13}+A^{-1})(3,-1)_T
\\ -A(5,-1)_T+A^{-8}(7,0)_T+(-2A^{-8}+A^{-4}-A^4)(5,0)_T
\\ +(-A^{-12}+A^{-8}-A^{-4}-1+A^4)(3,0)_T+(A^{-12}-A^{-8}+1+A^4)(1,0)_T 
\end{eqnarray*}
is in the peripheral ideal as well. Now, the Newton polygon is the convex hull of the points 
$$(7,0),(5,-1),(1,-2),(-3,-2),(-7,-1),(-7,0),(-5,1),(-1,2),(3,2) \ \text{and}\ (7,1).$$ The corner coefficients are again invertible and the slopes of the polygon are $0,2,4, -4$ and $\infty$. 

Moreover, by \cite[Lemma 5.1]{GS04}, we have that
$$\left((5,0)_T+(3,0)_T-(A^4+A^{-4})(1,0)_T\right)\cdot Y \in S(T^2)\emptyset$$
$$\left((5,0)_T+(3,0)_T-(A^4+A^{-4})(1,0)_T\right)\cdot Z \in S(T^2)\emptyset$$
The corresponding Newton polygons $P_2$ and $P_3$ have slope $\infty$ and corner coefficients $\pm A^k.$

Therefore, by Theorem \ref{thm:Dehn-filling_fin_gen}, Dehn fillings of the figure-eight knot of slopes different than $0,\pm 4$ and $\infty$ have their skein module finitely generated over $\Z[A^{\pm 1}]$. Of course, the Dehn filling of slope $\infty$ is $S^3$, so one can exclude the slope $\infty$ from the exception list.

Similarly, the skein module of the complement of the $(2,2n+1)$-torus knot is generated over the boundary by $\emptyset$ and $n$ other links, denoted by $y,y^2,\ldots,y^n$ in \cite{GS03}.

It is proved in \cite[Lemma 5.1]{GS03} that the following element is in the peripheral ideal of the $(2,2n+1)$-torus knot for $n\geq 0$:
$$z=(-2n-3,-1)_T+A^{-8}(-2n+1,-1)_T+A^{2p-5}(2n+3,0)_T -A^{2n-1}(2n-1,0)_T$$
The Newton polygon of $z$ is the convex hull of the points 
$$(2n+3,0),(-2n+1,-1),(-2n-3,-1),(-2n-3,0),(2n-1,1),(2n+3,1)$$ 
and all corner coefficients are invertible and the slopes are $4n+2,0$ and $\infty$ for the $(2,2n+1)$-twist knot. 

Moreover, by \cite[Proposition 4.1]{GS03}, for any $1\leq p \leq n,$ we have

$$(-1)^{p+1}(A^{-2p-4}S_{2p+2}(m)-A^{-2p}S_{2p-2}(m))y^p \in S(T^2)\emptyset +\ldots + S(T^2)y^{p-1},$$
where $m$ denotes the meridian of the $(2,2n+1)$-torus knot, and $S_k(X)\in \Z[X]$ is the Chebychev polynomial of the second kind
defined in Section \ref{sec:dimension},  which satisfies the same recurrence relation as $T_k(X),$ but with $S_0(X)=1, S_1(X)=X.$ 
Moreover, we note that $S_k(X)$ has dominant coefficient $1$ and degree $k,$ and that we have
$$S_{2k}(X)=T_{2k}(X)+T_{2k-2}(X)+ \ldots +T_2(X)+1.$$

Therefore, the associated Newton polygons have again invertible coefficient at the corners, and only $\infty$ as slope.

We can thus apply Theorem \ref{thm:Dehn-filling_fin_gen}. We conclude that all Dehn fillings of $(2,2n+1)$-torus knots for $n\geq 0$, except for the slopes $0$ and $4n+2$, have their skein modules finitely generated over $\Z[A^{\pm 1}]$. 
  
  Note that the Dehn $p/q$-filling of $E_{(2,2n+1)}$ is the mirror image of the Dehn $-p/q$-filling of $E_{(2,-(2n+1))}$
and their skein modules are related by the involution $A\to A^{-1}$. Therefore, the statement holds for negative $n$ as well.
\end{proof}

\section{Reducedness of Character varieties}
\label{sec:reduceness}

The goal of this section is  to prove  a reducedness criterion for character varieties of 3-manifolds obtained by surgeries on knots. We precede the statement and the proof of our criterion with some necessary preliminaries from algebraic geometry and representation theory of knot groups. Readers eager to move to applications of our criterion to skein modules may skip the proofs of the results and move to Section \ref{sec:reduceness_41} with out disruption of continuity.

\subsection{Preliminaries} All algebraic sets in this paper are affine, over $\C$. We denote them by capital letters and denote the coordinate ring of $X$ by $\C[X].$ Calligraphic letters, eg. $\mathcal Y,$ denote algebraic varieties, (see for example \cite[Ch. 5]{milne}) which are also affine and over $\C$, however may be non-reduced. For consistency, we denote their algebras of global sections by $\C[\mathcal Y].$ That can be any finitely generated commutative  $\C$-algebra. Dually, any scheme in this paper is secretly ${\rm Spec}\, A$, for some finitely generated $\C$-algebra $A$. Any scheme $\mathcal Y$ can be reduced to an algebraic set which we denote by $Y$. Then
$\C[Y]=\C[\mathcal Y]/\sqrt{0}.$ Conversely, every algebraic set $Y$ can be considered as a reduced scheme 
(and we use the same symbol for it.)

Note that for any knot $K$ and  a slope $p/q\in \Q\cup \{\infty\}$, representing a simple closed curve on $\partial E_K$, the character variety $X(E_K(p/q))$ is a closed subvariety of $X(E_K)$ cut by the equations $t_{\gamma m^p\ell^q}-t_\gamma$ for $\gamma\in \pi_1(E_K).$

Consider $\C^*\times \C^*$
 with coordinate system $(\mu, \lo)$ and let $\tau$ be an involution on $\C^*\times \C^*$ sending $(\mu, \lo)$ to $(\mu^{-1}, \lo^{-1})$.
Then the character variety of the 2-torus $T^2$, is $X(T^2):= (\C^*\times \C^*)/\tau.$ Denoting a set of generators of $\pi_1(T^2)=\Z^2$ by $m,\ell$ the above identifies $t_{m^p\ell^q}: X(T^2)\to \C$ with the function $\mu^p\lo^q+\mu^{-p}\lo^{-q}$ on $(\C^*\times \C^*)/\tau$ for any $p,q\in \Z.$

In what follows we use the canonical map $X(M)\to \cX(M)$ to identify points and irreducible components of $X(M)$ with the closed points and irreducible components of $\cX(M).$ 

Given a knot $K$, we identify $\p E_K$ with $T^2$ so that $m$ and $\ell$ correspond to a pair of  the meridian and the canonical  longitude of $K$. Then the natural embedding $T^2=\p E_K\hookrightarrow E_K$  defines a morphism $r: X(E_K)\to X(T^2).$ 

Given  a polynomial $f\in \C[x_1,...,x_n]$, let $V(f)$ denote its zero set of $f$ in $\C^n.$

The A-polynomial of a knot $K$, denoted by $A_K$, describes the image $r(X(E_K))$ in $X(T^2)$, which is know to consist of $0$- and $1$-dimensional components
 \cite{A-poly}.  However, no example of $K$ with $r(X(E_K))$ having a $0$-dimensional component is known. Let $Y$ be $r(X(E_K))$ with the $0$-dimensional components removed.
Then $A_K\in \C[\mu^\pmo, \lo^\pmo]$ is defined as the polynomial without repeated irreducible factors whose zero set $V(A_K) \subset\C^*\times \C^*$ is the algebraic closure of the preimage $\widetilde Y\subset \C^*\times \C^*$ of $Y\subset X(T^2)$ under the quotient 
$$\pi: \C^*\times \C^*\to (\C^*\times \C^*)/\tau = X(T^2).$$
Hence the polynomial $A_K(\mu, \lo)$ is describing the $r$ image of the non-abelian components of $\cX(E_K)$ in $X(T^2)$ lifted to $\C^*\times\C^*$\cite{A-poly}.
More generally, for any irreducible component $\cal C \subset \cX(E_K)$ 
there is the $A$-polynomial $A_{K,\cal C}$ describing $r(\cal C)$, as long as it is not a point.

\subsection{Reducedness criteria} 
 The goal of this section is prove the following generalization of Theorem \ref{t.reduced-i} of the Introduction:

\begin{theorem} \label{t.reduced}
Consider a knot $K$ and a slope $p/q\in \Q$ such that each closed point $\chi\in \cX(E_K(p/q))$ belongs to a unique irreducible component of $\cX(E_K)$, denoted by $\cC_\chi$.
Suppose further that for any such $\chi$,
\begin{enumerate}[(a)]
\item the map $r$ restricted to an open neighborhood of $\chi$ in $\cC_\chi$ is isomorphism onto its image
\item the polynomial  $A_{K, \cC_\chi}(x^{-q},x^{p})$ has no multiple roots, except  possibly $\pm 1$, 
\item if $(\chi(\ell),\chi(m))\in \{\pm 2\}^2$ then $\chi$ is abelian.
\end{enumerate}
Then $\cX(E_K(p/q))$ is finite and reduced.
\end{theorem}

Note that condition ($a$) of Theorem \ref{t.reduced}  holds for the abelian component of the character variety of every knot.
It is known that the components $\cal C$ of $X(E_K)$ are at least one dimensional while, as said above, $r(\cal C)$ must be at most $1$-dimensional. Therefore condition (a) implies that $dim\, r(\cC_\chi)=1$ and, consequently, $A_{K, \cC_\chi}$ is well defined. Condition (b) is satisfied for example when $A_{K}(x^{-q},x^{p})$  has no multiple roots, except possibly  $\pm 1$. Hence  Theorem \ref{t.reduced-i} is implied by
 Theorem \ref{t.reduced}.

Before we can complete the proof of Theorem \ref{t.reduced} we need some preparation and a few auxiliary lemmas.

Let $\cH(p/q):=\text{Spec}\, \C[\mu^\pmo, \lo^\pmo]/(\mu^p\lo^q-1).$ It is easy to check that $\dim\, T_\chi\, \cH(p/q)=1$ for all closed points $\chi$ of $\cH(p/q)$ and, hence, $\cH(p/q)$ is a smooth, reduced curve. By our convention, we will write $H(p/q)= \cH(p/q)$ when convenient, where 
\beq\label{e.H}
 H(p/q)=\{(\m,\lo)\in \C^{*}\times \C^{*}\ |\ \m^p\lo^q=1\}.
 \eeq 
 
Suppose that  two $1$-dimensional subvarieties $V,V'\subset \C^*\times \C^*$ intersect at $\chi$. We say that this intersection is \underline{transverse} if $V$ and $V'$ are smooth at $\chi$ and $T_\chi\, V+T_\chi\, V'=T_\chi\, (\C^*\times \C^* )$.

For $(p,q)\in \Z^2\setminus\lbrace (0,0)\rbrace$ coprime, the variety $H(p/q)$ is parametrized by $(\mu,\lambda)=(x^{-q},x^p)$,  where $x\in \C^*$. 
Hence, for any Laurent polynomial $A\in \C[\mu^\pmo,\lo^\pmo]$, the set $V(A)$ intersects $H(p/q)$ at points $(x^{-q},x^p)$ where $x$ goes over all roots of $A(x^{-q},x^p).$

\begin{lemma}\label{l.trans}
For $A\in \C[x^\pmo, y^\pmo]$, let  $(x^{-q},x^p)\in \C^*\times \C^*$ be a  point in $V(A)\cap H(p/q)$.\
Then, $V(A)$ intersects $H(p/q)$ transversely at $(x^{-q},x^p)$,
iff $x$ is a simple root of $A(x^{-q},x^p)$. 
\end{lemma}

\begin{proof}
By differentiating the defining polynomial $\mu^p\lambda^q-1$ of $H(p/q)$ and the defining polynomial $A$ of $V(A),$ we see that the tangent space of $V(A)\cap H(p/q)$ at $(x^{-q},x^p)$ is the set of $(u,v)\in \C^2$ such that
$$\begin{cases} \frac{\partial A}{\partial \mu}(x^{-q},x^p) u + \frac{\partial A}{\partial \lambda}(x^{-q},x^p) v=0
\\ p x^{-q(p-1)}x^{pq} u + q x^{-qp} x^{p(q-1)}v=0
\end{cases}$$
Taking the determinant, the tangent space is trivial if and only if
$$q x^{-p}\frac{\partial A}{\partial \mu}(x^{-q},x^p)-px^q\frac{\partial A}{\partial \lambda}(x^{-q},x^p)\ \ne 0, $$
which is equivalent to $A$ having a simple root at $(x^{-q},x^p).$
\end{proof}

Given an affine algebraic set $X$ and two algebraic subsets $X_1$ and $X_2$,
let $\cX_0$ denote their intersection in the sense of algebraic varieties. That is, if 
$\C[X_i]=\C[X]/I_i$ ($i=1,2$), then  $\C[\cX_0]=\C[X]/(I_1+I_2)$. Note that $\cX_0$ may be non-reduced, even though $X_1$ and $X_2$ are reduced by default.

The second lemma we need for the proof of Theorem \ref{t.reduced} is the following:

\blem\label{l.reduced}
Let $X_1,X_2\subset X$ and $\cX_0:=X_1\cap X_2$ be as above.
Then, any point $x_0\in \cX_0$ such that $T_{x_0}\, X_1\, \cap\, T_{x_0}\, X_2=0$ is an isolated reduced point of $\cX_0$.
\elem

\bpr  The closed immersions $\cX_0\to X_i$ of varieties for $i=1,2$, induce embeddings $T_{x_0}\, \cX_0\to T_{x_0}\, X_i$. 
Hence, $T_{x_0}\, X_1\, \cap\, T_{x_0}\, X_2=0$ implies $T_{x_0}\, \cX_0=0.$
\epr

We are now ready to give the proof of Theorem \ref{t.reduced}.

\noindent{\bf Proof of Theorem \ref{t.reduced}:}
Let $K$ be a knot and $p/q\in \Q$ a slope so that  the assumptions of Theorem  \ref{t.reduced} are satisfied.
 
 Since $\cX(E_K(p/q))$ is a closed subscheme of ${\rm Spec}\, \C[\cX(E_K)]/(t_{m^p\ell^q}-2)$,
 it is enough to show that  each  point of $\cX(E_K(p/q))$ is isolated and reduced in ${\rm Spec}\, \C[\cX(E_K)]/(t_{m^p\ell^q}-2)$.
 
 By the assumption of the statement, each such point $\chi$ belongs to a unique component $\cC_\chi$ of $\cX(E_K)$. By condition (a),  an open neighborhood of $\chi$ in $\cC_\chi$ is isomorphic to its image under $r$ in $r(\cC_\chi)$.

Therefore, 
it is enough to show that $r(\chi)$ is isolated and reduced as a point of the intersection $r(\cC_\chi)\cap X(p/q)$ (in the sense of algebraic varieties), where
$$X(p/q):= V(t_{m^p\ell^q}-2)\subset X(T^2).$$ 
In other words,  $X(p/q)$ is the image of $H(p/q)$ under the projection $\pi: \C^*\times \C^*\to X(T^2)$.

 By Lemma \ref{l.reduced}, it is enough to show that 
$$T_{r(\chi)}\, r(\cC_\chi)\cap T_{r(\chi)}\, X(p/q)=0.$$
If $\cC_\chi$ is the component of abelian characters in $\cX(E_K)$ then $\cC_\chi\simeq \cX(\Z)$ and 
the fact that $\chi$ is isolated and reduced follows from the fact that $X(\Z/p)$ is finite and reduced, cf. \cite[Thm 7.3(1)]{PS00}.

Hence, we can assume that $\cC_\chi$ is non-abelian, in which case, by condition (c), we have  $(\chi(\ell),\chi(m))\not\in \{\pm 2\}^2$.
Consequently, $r(\chi)$ is not a branching point of $\pi: \C^*\times \C^*\to X(T^2)$. 

Let $\bar\chi$ be a pre-image of $r(\chi)$ under $\pi.$ By the above argument, $\pi$ is \'etale in a neighborhood of $\bar\chi$. By definition, $\pi^{-1}(r(\cC_\chi)) \subset V(A_{K,\cC_\chi})$. Note also that the pre-image of $X(p/q)\subset X(T^2)$, under the  map $\pi: \C^*\times \C^*\to (\C^*\times \C^*)/\tau = X(T^2)$
 is $H(p/q)\subset \C^*\times \C^*$. Since $\pi$ induces an isomorphism of Zariski tangent spaces outside of branch points, by the above discussion it is enough to show that \beq\label{e.trans}
T_{\bar\chi}\, V(A_{K,\cC_\chi})\cap T_{\bar\chi}\, X(p/q)=0.
\eeq

By condition (b) in the statement of Theorem \ref{t.reduced}, $A_{K,\cC_\chi}(x^{-q},x^q)$ is not the zero polynomial
and, hence, by Lemma \ref{l.trans}, $V(A_{K,\cC_\chi})$ intersects $H(p/q)$  transversally at $(x^{-q},x^q)$.
Consequently,
$$T_\chi\, (V(A_K)\cup H(p/q)) = T_{(x_\chi^{-q},x_\chi^p)}\, (\C^*\times C^*),$$
which indeed implies \eqref{e.trans}.
\qed
\vspace*{.1in}

The following provides a reformulation of condition (a) of Theorem \ref{t.reduced}, which we will be useful for us later on.

\bpro\label{p.birat} With the notation and setting  of Theorem \ref{t.reduced},
condition ($a$) of is equivalent to the following:

($a'$) $r$ is a birational map of $\cC_\chi$ onto its image.
\epro

\bpr
$(a') \Rightarrow (a)$:  Since $r$ is defined on $\chi$ it is an isomorphism of an open neighborhood $U$ of $\chi$ onto $r(U).$\\
$(a) \Rightarrow (a')$:  The proof of Theorem \ref{t.reduced} shows that conditions ($b$) and ($c$) of Theorem \ref{t.reduced} imply that
$r$ maps $X(E_K(p/q))$ to smooth points of $r(X(E_K(p/q))$. Now \cite[I.6.8]{HartshorneAG} implies that $r$ is isomorphism of an open neighborhood of $X(E_K(p/q))$ (onto its image).
\epr

 We call the images of torsion points under the map $\pi: \C^*\times \C^*\to (\C^*\times \C^*)/\tau = X(T^2)$
 the \underline{torsion points} of $X(T^2).$

\begin{theorem} \label{t.MM} 
Suppose that for a knot $K$,
\begin{enumerate}[(a)]
\item  $r$ is birational on every irreducible component of $\cX(E_K)$ onto its image in $X(T^2),$  
\item $r$ maps singular points of $X(E_K)$ to non-torsion points in $X(\p E_K)$, and
\item every $\chi\in \cX(E_K)$ with $(\chi(\ell),\chi(m))\in \{\pm 2\}^2$ is abelian.
\end{enumerate}
Then, for all but finitely many slopes $p/q$, the character scheme $\cX(E_K(p/q))$ is finite and reduced.
\end{theorem}

\begin{proof}
The proof follows that of Theorem \ref{t.reduced}. Since intersection points of irreducible components of $X(E_K)$ are singular in $X(E_K)$, condition (b) implies that each point of $X(E_K(p/q))$ belongs to a unique irreducible component of $X(E_K)$ for almost all slopes $p/q.$ By the same condition, the points of $V(A_K)\cap H(p/q)$ are smooth in $V(A_K)$ for almost all slopes $p/q.$ We will consider those slopes $p/q$ only for which these conditions hold. By \cite[Theorem 3.2]{MM22}, $V(A_K)$ intersects $H(p/q)$ transversally for all but finitely many $p/q$. Note that their statement is for irreducible $A$ but it generalizes to any $A$ without repeated factors.

By Proposition \ref{p.birat}, condition (a) above implies condition (a) of Theorem \ref{t.reduced}.
Now the statement follows from Lemma \ref{l.reduced}, along the lines of the proof of Theorem \ref{t.reduced}.
\end{proof}

\section{Dehn fillings of the figure-eight knot}
\label{sec:reduceness_41}
In this section we study the  character schemes of 3-manifolds obtained by Dehn filling on the figure-eight knot $K=4_1$ and we prove  part (a) of Theorem \ref{t.dimensions} stated in the Introduction.
\subsection{Reducedness of character schemes}As discussed, for example, in \cite{HS},
the
 A-polynomial of the figure-eight knot is
$$A_{4_1}(\mu, \lo) =-\lo+\lo\cdot \mu^2 +\mu^4 +2\lo\cdot \mu^4 +\lo^2\cdot \mu^4 +\lo\cdot \mu^6 -\lo\cdot \mu^8.$$ 
By substituting $\mu=x^{-q}, \lo=x^p$ we get
$$A_{4_1}(x^{-q},x^p)=-x^{p}+x^{p-2q}+x^{-4q}+2x^{p-4q}+x^{2p-4q}+x^{p-6q}-x^{p-8q}.$$
This last polynomial multiplied by $x^{4q}$ appears in \cite[Prop. 5.8]{ChM}.

Our goal in this subsection is to prove prove two results. The first result is the following:

\begin{theorem}\label{t.41reduced}
If $A_{4_1}(x^{-q},x^p)$ has no multiple roots other than $\pm 1$, then $\cX(E_{4_1}(p/q))$ is finite and reduced.
\end{theorem}

The proof of Theorem \ref{t.41reduced} requires a preliminary discussion that we will postpone it after the proof of the second main result of this subsection concerning roots of $A_{4_1}(x^{-q},x^p)$.

We recall that, as observed in \cite[Sec. 5]{ChM}, the polynomial $A_{4_1}(x^{-q},x^p)$ has double root  $x=-1$ when $p$ is odd. We show the following:

\bpro\label{p.rootsp1}
 The only non-simple root of $A_{4_1}(x^{-q},x)$ is $x=-1$. Consequently, by Theorem \ref{t.41reduced},
 $\cX(E_{4_1}(1/q))$ is  finite and reduced.
\epro

\bpr Let
$$P=\frac{x^{8q-1}\cdot A_{4_1}(x^{-q},x)}{(x+1)^2}.$$
By a direct computation, we have
$$P=x^{4q-1}-  (x^{4q}+x^{2q}+1)\cdot \left(\frac{x^{2q}-1}{x+1}\right)^2.$$
We need to prove that $P$ has no multiple roots. That can be checked for $q=\pm 1$ by direct computation. Hence, assume that $q\ne \pm 1$ from now on.

Let us compute its derivative $P'$ mod $q$. Since $(x^{4q}+x^{2q}+1)'=0$,
$$\left(\frac{x^{2q}-1}{x+1}\right)'=\frac{0-(x^{2q}-1)\cdot 1}{(x+1)^2}=- \frac{x^{2q}-1}{(x+1)^2},$$
$$\left(\left(\frac{x^{2q}-1}{x+1}\right)^2\right)'=-2 \left(\frac{x^{2q}-1}{x+1}\right)\cdot \frac{x^{2q}-1}{(x+1)^2},$$

and
$$P'=-x^{4q-2} +  2(x^{4q}+x^{2q}+1)\cdot (x^{2q}-1)^2/(x+1)^3.$$

Then 
\begin{equation}\label{e.PP'}
2P-(x+1)P'=(x-1)x^{4q-2} \ \text{mod}\ q.
\end{equation}

Since $q\ne \pm 1$, it has a prime divisor $d$.
Let us assume that $\alpha\in \C$ is a multiple root of $P$. Since $P$ is monic, $\Z[\alpha]$ projects onto
a non-zero $\F_d$-vector space, $\Z[\alpha]\otimes \F_d.$ Let $\overline\alpha$ be the image of $\alpha$ under that projection. Then $P, P'$ considered as elements of $\F_d[\overline\alpha][x^\pmo]$ have a common divisor $(x-\overline\alpha)$. 

Then $P$ and $P'$ have a common divisor $(x-\overline\alpha)$ in 
$\F_d[\alpha][x],$ where $\overline\alpha$ is the image of $\alpha$ in $\F_d[\alpha]$. 
Since $(x-\overline\alpha)$ divides \eqref{e.PP'} in $\F_d[\overline\alpha][x^\pmo]$, we have $\overline\alpha=0$ or $1$ in $\F_d[\alpha]$. That means that
$0$ or $1$ is a root of $P$ mod $d$.
But that is not the case. Indeed, $\frac{x^{2q}-1}{x+1}$ is $-1$ for $x=0$ and it vanishes for $x=1$ (mod $d$).
Consequently, $P(0)=-1$ for $x=0$ and it is $1$ for $x=1$ (mod $d$).
\epr

Let us now discuss properties $SL(2,\C)$-representations of $\pi_1(E_{4_1})$ with the goal of proving Theorem \ref{t.41reduced}.

The fundamental group of the figure eight knot has a presentation $\la a,b\, |\, aW=Wb\ra,$ where $a,b$ are generators associated with the two bridges in the standard realization of $4_1$ as a $2$-bridge knot and $W=b^{-1}aba^{-1}$. See, for example,  \cite[Prop. 1]{HS}, where $4_1$ is denoted by $J(2,-2)$.
A direct computation (for example using a computer algebra system) shows that, for $\mu\in \C^*$ and $\tau\in \C$, the assignment
$$\psi_{\mu, \tau}(a)= \left(\begin{matrix} \mu & 1 \\ 0 & \mu^{-1}\\ \end{matrix}\right) \ \text{and}\ \psi_{\mu,\tau}(b)= \left(\begin{matrix} \mu & 0 \\ \tau & \mu^{-1}\\ \end{matrix}\right),$$
defines a representation $\psi_{\mu, \tau}: \pi_1(E_{4_1}) \longrightarrow SL(2,\C)$,
 if and only if we have
\begin{equation}\label{e.zm}
\tau^2+(3-\mu^2-\mu^{-2})(1-\tau)=0.
\end{equation}
 (We note  that our $\tau$ corresponds to $-z$ in \cite{HS}.)
 By another calculation, we see that  $Tr([\psi_{\mu,\tau}(a),\psi_{\mu,\tau}(b)])=2$ at $\tau=2-\mu^2-\mu^{-2}$ or at $\tau=0$.
However, since the value $\tau=2-\mu^2-\mu^{-2}$  is impossible by \eqref{e.zm}, we obtain that
\beq\label{e.trcomm}
 Tr([\psi_{\mu,\tau}(a),\psi_{\mu,\tau}(b)])\ne 2\quad \ \text{for}\ \tau\ne 0.
 \eeq
Hence, by \cite[Lemma 1.2.1]{CullerShalen}, the representation
$\psi_{\mu,\tau}$ is irreducible for
 $\tau\neq 0$.

Each irreducible $SL(2,\C)$-representation of $\pi_1(E_{4_1})$ is conjugate to $\psi_{\mu,\tau}$ for a unique $\tau\ne 0, 2-\mu^2-\mu^{-2}$. Indeed, $a$ and $b$ are conjugated in our presentation of $\pi_1(E_{4_1})$, and one can choose a basis of $\C^2$ consisting of an $\mu$-eigenvector of $\psi_{\mu,\tau}(a)$ and of an $\mu^{-1}$ eigenvector of $\psi_{\mu,\tau}(b)$. There is exactly one such basis for which the upper right coefficient of $\psi_{\mu,\tau}(a)$ is $1$.

This shows that $X(E_{4_1})$ consists of the abelian irreducible component $X^{ab},$ and of a single non-abelian component $X^{na},$ defined by \eqref{e.zm}:

$$X^{na}=\lbrace (\mu,\tau)\in \C\times \C \ | \ \tau^2+(3-\mu^2-\mu^{-2})(1-\tau)=0 \rbrace.$$
 The non-abelian  component contains the $SL(2,\R)$ lifts of the monodromy of $PSL(2,\R)$-representation of hyperbolic structures on $E_{4_1}$.

Let $\Tr\, \psi_{\mu,\tau}(\ell):=\lo+\lo^{-1}$, where $\ell$ represents the canonical   longitude of $E_{4_1}$. Then, by  \cite[Eg (5.9)]{HS} for instance, we have
\beq\label{e.41tau}
\tau(\lo+\mu^2)=(\mu^2-1)(1-\lo).
\eeq
\vs

We now give the proof of Theorem \ref{t.41reduced}.

\begin{proof}[Proof of Theorem \ref{t.41reduced}]  The proof relies on Theorem \ref{t.reduced}.
In \cite[section 5]{LeT} it is proved that $\cX(K)$ is reduced for all $2$-bridge knots and, hence, we will consider $X(E_{4_1})$ instead of $\cX(E_{4_1})$ in the proof.

By the above discussion, any character in $X^{ab}\cap X^{na}$ is represented by $\psi_{\mu,\tau}$ where $\tau=0$ and, by \eqref{e.zm}, $\mu$ must be irrational. On the other hand, \eqref{e.41tau} forces $\lo=1.$
Therefore, $\mu^p\lo^q\ne 1$ for coprime $p,q$, implying that every point $X(E_{4_1}(p/q))$ belongs to a unique irreducible component  $\cC_\chi$ of $X(E_{4_1})$. By our earlier discussion either   $\cC_\chi=X^{ab}$ or $\cC_\chi=X^{na}$.

Next we check  that each $\chi\in X(E_{4_1}(p/q))$ satisfies conditions (a)-(c) of Theorem \ref{t.reduced}.\vs

(a) Since this condition holds for the abelian component, assume that $\cC_\chi=X^{na}.$ 
Let 
\begin{equation}\label{e.mapfi}
\phi: V(A_{4_1})-\{\lo=\pm 1\}\to X^{na}\subset X(E_{4_1}),
\end{equation}
send $(\mu, \lo)\in V(A_{4_1})$ to $\psi_{\mu, \tau}$ where  
$$\tau=\frac{(\mu^2-1)(1-\lo)}{\lo+\mu^2}.$$
The map $\phi$ is well defined.  Indeed, if $\lo+\mu^2= 0$, then, by Equation
\eqref{e.41tau},
 $(\mu^2-1)(1-\lo)=0$ and hence
$\lo=\pm 1.$ 
Since $t_m=\mu+\mu^{-1}$ and $\tau=2-t_{ab^{-1}}$, the map $\phi$ is a morphism of varieties. By \eqref{e.41tau}, it is the inverse to $r$.

(b) follows from the assumption of Theorem \ref{t.41reduced}.

(c) The values $\chi(\ell),\chi(m)\in \{\pm 2\}$ correspond to $(\mu,\lo)\in \{\pm 1\}^2$, for which any
corresponding  non-abelian $SL(2,\C)$-representation $\psi_{\mu, \tau}$ of $\pi_1(E_{4_1})$ can be conjugated to $\psi'$ so that
$$\psi'(m)= \left(\begin{matrix} \mu & 1 \\ 0 & \mu^{-1}\\ \end{matrix}\right) \ \text{and}\ \psi'(l)= \left(\begin{matrix} \lo & z \\ 0 & \lo^{-1}\\ \end{matrix}\right),$$
for some $z\in \C.$
This representation factors through the parabolic ${\rm PSL}(2,\C)$-representation, which must be the monodromy of the complete hyperbolic structure of the complement of the figure-eight knot. For such representation $1$ and $z$ are $\R$-linearly independent, as the representation is discrete and faithful. Hence, this representation cannot satisfy $\psi'(m)^p\psi'(\ell)^q={\rm I}.$ Indeed, note that 
$$\psi'(m)^p=\mu^p \begin{pmatrix}
	1 & p\mu
	\\ 0 & 1
\end{pmatrix}, \psi'(\ell)^q=\lambda^q \begin{pmatrix}
1 & q\lambda z
\\ 0 & 1
\end{pmatrix},$$
since $\mu,\lambda\in \lbrace \pm 1 \rbrace.$  If $\psi'(m)^p\psi'(\ell)^q=I$ then one must have that $p\pm q\cdot z=0$, which is impossible since $p,q$ are co-prime.
This shows that condition (c) of Theorem \ref{t.reduced} is satisfied.
\end{proof}

\subsection{Dimensions under Dehn fillings} 
For relatively prime integers let $M=E_{4_1}(p/q)$ denote the 3-manifold obtained by 
 $p/q$-Dehn filling of the figure-eight knot complement $4_1$.

Recall that
$$d_{4_1}(p/q):= {1\over2} (|4q+p| + |4q-p|)- \delta_{2\nmid p},$$
where $\delta_{2\nmid p}=1$ for odd $p$, and 0 otherwise.

Next we prove the following theorem which immediately  implies, in particular, part (a) of Theorem \ref{t.dimensions}:
\begin{theorem}\label{t.X41} For any slope $p/q\not\in \{0, \pm 4\}$, and $M:= E_{4_1}(p/q)$,
$$\dim_{\Q(A)}S(M)\geq d(p/q)+1+\left\lfloor\frac{|p|}{2}\right\rfloor.$$
Furthermore, with the exception of finitely many slopes $p/q$,  $\cX(M)$ is reduced and
$$\dim_{\Q(A)}S(M)=|X(M)|=d(p/q)+1+\left \lfloor\frac{|p|}{2}\right \rfloor.$$
\end{theorem}

We will construct specific bases of skein modules $S(M)$ for slopes for which $\cX(M)$ is reduced in Section \ref{sec:bases_41}. In contrast to Theorem \ref{t.41reduced}, the statement of Theorem \ref{t.X41} does not provide concrete slopes for which the equality of the statement holds since the excluded slopes are not explicit.

Let $X^{irr}(M)$ the set of irreducible characters in $X(M)$.

\begin{proof}[Proof of Theorem \ref{t.X41}:] The proof utilizes the computation of the $SL(2,\C)$-Casson invariant $\lambda_{SL(2,\C)}(M)$, given in  \cite{BC06}. 
By Corollary 2.2 and Theorem 5.7 in \cite{BC06} (where $4_1$ corresponds to $\xi=2$), 
$$|X^{irr}(E_{4_1}(p/q)))|=\lambda_{SL(2,C)}(E_{4_1}(p/q))=d(p/q),$$
for $p/q\neq \pm4$. 

Since the number of reducible characters in $X(E_{4_1}({p/q}))$ is $1+\left \lfloor\frac{|p|}{2}\right \rfloor$, we have 
$$|X(E_{4_1}({p/q}))|=d(p/q)+1+\left \lfloor\frac{|p|}{2}\right \rfloor.$$

Since $S(E_{4_1}({p/q}))$ is tame by Theorem \ref{t.fin_gen_skein}(a), the inequality of the statement follows from Theorem  \ref{t.main}(a).

To prove the equality, we apply Theorem \ref{t.MM} to show that with the exception of finitely many slopes $p/q$,  $\cX(E_{4_1}(p/q))$ is reduced. Note that the map
$\phi$ defined in \eqref{e.mapfi} in the proof of Theorem \ref{t.41reduced}, is a birational morphism from $X^{na}$ onto its image in $X(T^2).$ Therefore condition (a) is satisfied for all characters of $E_{4_1}(p/q).$ The proof of Theorem \ref{t.41reduced} shows that the condition (c) is satisfied as well.

Therefore, it remains to check condition ($b$). As discussed earlier, $X(E_{4_1})$ consists of the abelian $X^{ab}$ and non-abelian components $X^{na}$. 

The only singular point of $X(E_{4_1})$ in
$X^{ab}$ is at the intersection of $X^{ab}$ and $X^{na}$, which by \eqref{e.trcomm} corresponds to $\tau=0,$
which by \eqref{e.zm} corresponds to $\mu$ satisfying $\mu^2+\mu^{-2}=3$, and, hence, are not roots of $1$.

It remains to be shown that the singular points in $X^{na}$ are non-torsion either. These singular points correspond to points where the defining equation \eqref{e.zm} and its partial derivatives are simultaneously
 zero. That is, setting 
 $$f(\tau, \m):=\tau^2+(3-\mu^2-\mu^{-2})(1-\tau),$$ 
 we must have
 $$f(\tau, \m)=\frac{\partial f}{\partial \mu}=\frac{\partial f}{\partial \mu}=0$$
 at singular points.
 Note that 
$$\frac{\partial f}{\partial \mu}=-2\m (1-\m^{-2})(1-\tau)\quad \text{and}\quad \frac{\partial f}{\partial \tau}=2\tau+\mu^2+\mu^{-2}-3.$$
Hence, $\frac{\partial f}{\partial \mu}=0$ implies that $\m=\pm 1$ or $\tau=1$. ($\mu\ne 0$ since it is invertible.)
Condition $\frac{\partial f}{\partial \tau}=0$ implies that $\tau=1/2$ in the first case, and
$\mu^2+\mu^{-2}=1$ in the second. Note that in each of these cases $f(\tau,\mu)\ne 0$ and, hence, $X^{na}$ has no singular points.
\end{proof}

\section{Computings bases of coordinate rings}
\label{sec:bases_41}

In this section we develop a method of finding bases for the coordinate rings $\C[X(E_K(p/q)]$. Then we apply it to computing bases for the coordinate rings of character varieties of 3-manifolds obtained by Dehn fillings of $E_{4_1}$.

\subsection{Bases of $\C[X(E_K(p/q)]$} 
Recall from Section \ref{sec:reduceness} the map
 $r: X(E_K)\to X(T^2)$ that  is induced by the inclusion $T^2=\p E_K\hookrightarrow E_K$.
  
Let $MV(K)$ denote the set of multiple-values of $r: X(E_K)\to X(T^2)$. That is
$$MV(K):=\{\chi\in r(X(E_K)): |r^{-1}(\chi)|>1\}\subset X(T^2).$$

Recall that $X(p/q)=V(t_{p/q}-2) \subset X(T^2)$ and that $t_{s/u}$ denotes the trace function of a simple closed curve on $T$ representing the slope $s/u\in \Q\cup\{1/0\}$.

We say that a Dehn filling slope $p/q$ of $E_K$ is \underline{excluded} if $X(E_K(p/q))$ is infinite
or if $r(X(E_K(p/q))\cap MV(K)\ne \emptyset$.

\begin{proposition}\label{t.basis-char}
For any non-excluded slope $p/q$ of a knot $K$ with finite $X(E_K(p/q))$, and for any slope $s/u$ such that  $pu-qs=1$,
the set 
$$B:=\{t_{s/u}^j \ |\  0\leq j <|X(E_K(p/q))|\},$$ 
is a basis of $\C[X(E_K(p/q))]$.
\end{proposition}

The proof of the proposition  shows more generally that the elements of $B$ are linearly independent, even when $X(E_K(p/q))$ is infinite.
We precede the proof with the following lemma:
\blem
If  $pu-qs=1$, then $t_{s/u}$ is $1$-$1$ on $X(p/q)$. 
\elem

\bpr
After fixing a basis $\m, \lo$ of $H_1(T^2)$  we identify the mapping class group of $T^2$ 
 with $SL(2,\Z)$. Its left action of the mapping class group on $T^2$ induces a right action on $X(T^2)$ in which $A\in SL(2,\Z)$ maps $t_{s/u}$ to $t_{s'/u'}$ where
$$\left(\begin{matrix} s' \\ u'\\ \end{matrix}\right)=A\cdot \left(\begin{matrix} s \\ u\ \end{matrix}\right).$$
 
Consequently, $A= \left(\begin{matrix} p & s\\q & u\\ \end{matrix}\right)^{-1}$ maps $t_{s/u}$ to $t_{0/1}$ and
it maps $X(p/q)$ to $X(1/0)$. 

Hence, it is enough to show that $t_{0/1}$ is $1$-$1$ on $X(1/0)$.
Since $t_{0/1}=\lo+\lo^{-1}$, its value determines $\lo$ up to inversion and $\lo^{\pm 1}$ defines a unique point $(1,\lo^{\pm 1})$ in $X(1/0)$. Here, we use the fact that
$$X(1/0)=(\{1\}\times \C^{*}\})/_\sim,$$
where $\sim$ identifies $(\m, \lo)$ with $(\m^{-1},\  \lo^{-1})$, as before. 
\epr

\begin{proof}[Proof of Proposition \ref{t.basis-char}:]
Since $p/q$ is a non-excluded slope, $r: X(E_K(p/q))\to X(T^2)$ is $1$-$1$,
and, hence, the trace function $t_{s/u}$ is $1$-$1$ on $X(E_K(p/q)).$

To finish the proof, we use the fact that the following Vandermonde matrix 
$$(t_{s/u}(\chi)^j) \ \text{for}\ \chi \in X(E_K(p/q)) \ \text{and}\ 0\leq j< |X(E_K(p/q))|.$$
has determinant 
$$d:=\pm\prod (t_{s/u}(\chi)-t_{s/u}(\chi')),$$
where the product is over all $2$-element subsets $\{\chi, \chi'\}\subset X(E_K(p/q))$.
Since
$t_{s/u}$ is $1$-$1$ on $X(E_K(p/q))$, we conclude that $d\neq 0$, which means that the elements of
$B$ are linearly independent.
Since the cardinality of $B$ coincides with the dimension of $\C[X(E_K(p/q))]$, the set $B$ is  a basis of $\C[X(E_K(p/q))]$.
\end{proof}
\medskip

Recall that a \underline{ boundary slope} of a knot $K$ is  the slope of each boundary component of a properly embedded incompressible surface in $(E_K, \partial E_K)$.
Also recall that if $X(E_K(p/q))$ is infinite, $E_K(p/q)$ contains a closed incompressible surface, cf. 
\cite[Theorem 2.2.1]{CullerShalen}. Hence,
if the only incompressible surfaces in $E_K$ are tori parallel to $\p E_K$, then excluded slopes of the first kind form a subset of the set of boundary slopes of $K$. The set of such slopes is finite by \cite{Ha}. 

\subsection{A basis for  $\C[X(E_{4_1}(p/q))]$}
The fundamental group of the figure-eight knot has a presentation 
$$\pi_1(E_{4_1})=\la a,b | aW=Wb\ra,$$ 
where $a,b$ are generators associated with the two bridges in the standard realization of $4_1$ as a $2$-bridge knot and $W=b^{-1}aba^{-1}$. 

\def\basisfeightstat{Let $p,q,s,u\in \Z$  such that $pu-qs=1$, $p/q \ne 0, \pm 4$ and $M:=E_{4_1}(p/q)$.
Then, a basis of $\C[X(M)]$ is given by 
\begin{enumerate}[(a)]
\item the set $B=\{t_{s/u}^j\  | \ 0\leq j < |X(M)|\}$, if $p$ is odd, or   $q$ is odd and $p\equiv 2 \ {\rm mod}\  4$,
\item the set 
$B=\{ t_{ab^{-1}}, \, t_{ab^{-1}}^2,\,  t_{ab^{-1}}t_{s/u}, \, t_{ab^{-1}}^2 t_{s/u}\rbrace \cup \lbrace t_{s/u}^j\ | \ 0\leq j< |X(M)|-4\}$, otherwise.
\end{enumerate}}

\begin{theorem}\label{t.basisf8}
\basisfeightstat

\end{theorem}

Given a slope $p/q$, let $\overline{MV}({4_1})$ denote the preimage of $MV(4_1)$ under the standard projection $\pi: \C^*\times \C^*\to X(T^2)$ and let
$I(p/q)$ denote the preimage of $MV(4_1)\cap r(X(E_K(p/q)))$ under $\pi$. Then 
$$I(p/q) \subset \overline{MV}({4_1})\cap H(p/q)\subset \C^*\times \C^*,$$
where $H(p/q)=\{(\m,\lo): \m^p\lo^q=1\}$, cf.  \eqref{e.H}.
For the proof of Theorem \ref{t.basisf8} we need the following lemma 
that computes the sets 
$I(p/q)$: 

\begin{lemma}\label{l.mult_values} 
 (a) The set  $\overline{MV}({4_1})$ consists of four points: $(\pm 1,-1)$ and $(\pm i, 1)$.\\ 
Furthermore,\\ 
(b) $I(p/q)=\{(i,1),(-i,1)\}$ for $4\mid p$ and $q$ odd,\\ 
(c) $I(p/q)=\emptyset$ otherwise. 
\end{lemma}

\begin{proof}
(a) Recall from Section \ref{sec:reduceness_41} that each irreducible representations of $\pi_1(E_{4_1})$ is conjugated to a unique representation $\psi_{\mu,\tau}$ introduced in Section \ref{sec:reduceness_41}, where $(\mu,\tau)$ satisfy Equation (\ref{e.zm}) and $\tau\neq 0,2-\mu^2-\mu^{-2}.$

By Equation (\ref{e.41tau}), $\lo+\mu^2=0$ implies $(\mu,\lo)=(\pm 1,-1)$ or $(\pm i, 1)$.
 Therefore, by Equation (\ref{e.41tau}), $\tau$ is determined uniquely by $\mu$ and $\lo$, except possibly for the above values of $(\mu, \lo)$.

In those cases, Equation (\ref{e.zm}) has solutions $\tau=\frac{1\pm i\sqrt{3}}{2}$ and $\tau=\frac{5\pm \sqrt{5}}{2}$ respectively, and in both cases those two solutions correspond to irreducible representations since $\tau\neq 0, 2-\mu^2-\mu^{-2}.$  

Therefore, $r$ is $1$-$1$ on the irreducible characters of $\pi_1(E_{4_1}),$ except for the two double values corresponding to the above values of $\tau.$ 

Since $r$ is $1$-$1$ on the abelian characters, it remains to be shown that $r(\chi)\ne r(\chi')$ for an abelian character $\chi$ and an irreducible character $\chi'$ other than the ones corresponding to $(\mu, \lo)= (\pm 1, -1)$ and $(\pm i, 1).$ This is indeed the case, because $\lambda=1$ for all abelian characters and Equation (\ref{e.zm}) implies $\mu=\pm i$ in this case.

(b) Note that by the argument in the proof of Theorem \ref{t.41reduced}, $(\mu,\lo)=(\pm 1,-1)$ does not correspond to a representation of $\pi_1(E_{4_1}(p/q)).$ Therefore, $I(p/q)\subset \{(i,1),(-i,1)\}.$

Let $p\equiv 0$ mod $4$ now. Then $q$ is odd, by the coprimness of $p$ and $q$. To show that $I(p/q) = \{(i,1),(-i,1)\}$, we need to establish that these points correspond to a character $Tr\, \rho$ of $\pi_1(E_{4_1})$ which factors through $\pi_1(E_{4_1}(p/q)).$ To show it, note that $Tr\, \rho(m)=i-i=0$ and, hence, $\rho$ is diagonalizable.
Since $\rho(l)$ commutes with $\rho(m)$, the representation $\rho$ can be conjugated so that both $\rho(m), \rho(l)$ are diagonal. Since $4\mid p,$ we have $\m^p\lo^q=1$ and, hence, $\rho(m^p\ell^q)$ is the identity matrix.

(c) Since $\m^p\lo^q=1$ for all points of $X(E_K(p/q))$, the set $I(p/q)$ is empty unless $p\equiv 0$ mod $4$. 
\end{proof}

We are now ready to give the proof of Theorem  \ref{t.basisf8}.

\begin{proof}[Proof of Theorem \ref{t.basisf8}]
(a) Since the complement of $4_1$ contains no closed incompressible surfaces, 
such surfaces in $M=E_{4_1}(p/q)$ can only appear when $p/q$ is a boundary slope. 
Since the boundary slopes of $4_1$ are $0,\pm 4$, \cite[p. 231]{Hath}, $X(M)$ is finite by the assumption of the statement and by Culler-Shalen theory. 
Now the statement follows from Lemma \ref{l.mult_values}(c) and Proposition \ref{t.basis-char}.

 (b) Let us assume now that $p\equiv 0$ mod $4$ and $q$ is odd. Then $I(p/q)=\{(i,1),(-i,1)\}$ by Lemma \ref{l.mult_values}(b). Let us denote the abelian characters corresponding to these points by $\chi_+$ and $\chi_-$, respectively. By the discussion in the proof of that lemma, there are two additional non-abelian characters corresponding to each of these points, which we denote by $\chi_+',\, \chi_+''$ and by $\chi_-',\, \chi_-''$.

Since $t_{s/u}$ is $1$-$1$ on $Y:=X(M)-\{\chi_+', \chi_+'', \, \chi_-',\, \chi_-'' \}$, by
the Vandermonde matrix argument  of Proposition \ref{t.basis-char}, applied to $Y$ 
the elements of $B'=\{t_{s/u}(\chi)^j\}_{j=0}^{|Y|-1}$ are linearly independent. 

A direct computation shows that
$t_{ab^{-1}}=2-\tau$ on a representation $\rho_{\mu,\tau},$ and $t_{ab^{-1}}=2$ on abelian characters. Since 
$\chi_\pm'$ and $\chi_\pm''$ represent different $\tau$-values, $t_{ab^{-1}}$ takes three distinct values on $\chi_+,\, \chi_+',\, \chi_+''$ and three distinct values on $\chi_-,\, \chi_-',\, \chi_-''.$

Let $$B''=\{t_{ab^{-1}},\, t_{ab^{-1}}^2,\,  t_{ab^{-1}}t_{s/u},\, t_{ab^{-1}}^2t_{s/u}\}.$$
We claim that the functions in $B''$ are linearly independent on $\{\chi_+', \chi_+'', \, \chi_-',\, \chi_-'' \}$.
Indeed,
 $$(c\cdot t_{ab^{-1}}+d\cdot t_{ab^{-1}}^2)+t_{s/u}(e \cdot t_{ab^{-1}}+f\cdot t_{ab^{-1}}^2)=(c+e\mu^s)t_{ab^{-1}}+(d+f\mu^s)t_{ab^{-1}}^2$$
takes  the same value for $\chi_+',\, \chi_+''$ if and only if 
$$c+e i^s=d+fi^s=0,$$ and it takes the same value for $\chi_-',\, \chi_-''$ only if 
$$c+e (-i)^s=d+f(-i)^s=0.$$ Since $pu-qs=1$, $s$ is odd and the above happens only if 
 $c=d=e=f=0.$ 

Since 
$$t_{s/u}(\chi_+)=t_{s/u}(\chi_+')=t_{s/u}(\chi_+'')\quad \ \text{and}\quad t_{s/u}(\chi_-)=t_{s/u}(\chi_-')=t_{s/u}(\chi_-''),$$ 
no non-trivial linear combination of the functions in $B''$ lies in the span of the powers of $t_{s/u}.$
Consequently, the functions in $B=B'\cup B''$ are linearly independent and, since $|B|=|X(M)|$, they form a basis of $\C[X(M)].$ 
\end{proof}

\section{Dehn fillings of $(2,2n+1)$-torus knots.}
\label{ss.torusknot}

\subsection{Reducedness under Dehn filling} 
As before, let $K=T_{(2,2n+1)}$ denote the $(2,2n+1)$-torus knot, for any $n\in \Z$. 

 As in \cite[Example 1.24]{Ha},  $\pi_1(E_K)=\la a,b \ |\  a^2=b^{2n+1}\ra$, where
the meridian and longitude are given by 
\begin{equation}\label{e.ml}
m=a\cdot b^{-n}, \ \text{and}\ l=a^{2}\cdot m^{-4n-2}.
\end{equation}

The center of $\pi_1(E_K)$  is generated by $a^2=b^{2n+1}$. Hence, by Schur's lemma,  any irreducible $SL(2,\C)$-representation $\rho$ of $\pi_1(E_K)$ sends $a^2$ to
a scalar matrix, $c\cdot {\rm I}$. Furthermore, $c=-1$ since $\rho(a^2)={\rm I}$ implies that $\rho(a)=\pm{\rm I}$ and, hence, $\rho$ is reducible. 

Consequently, each irreducible representation of $\pi_1(E_K)$ is conjugate to $\rho_{\tau,\zeta}$, given by
\begin{equation}\label{e.torusrep}
\rho_{\tau,\zeta}(a)= \left(\begin{matrix} i & 1\\ 0 & - i\\ \end{matrix}\right) \ {\rm and} \  \rho_{\tau,\zeta}(b)=\left(\begin{matrix} \zeta & 0\\ \tau& \zeta^{-1}\ \end{matrix}\right),
\end{equation}
for some $\tau, \zeta\in \C$, where $\zeta^{2n+1}= -1$. 
 
It is easy to check that
$\rho_{\tau,\zeta}(b)^k= \left(\begin{matrix} \zeta & 0\\ \tau\cdot \phi_k(\zeta) & \zeta^{-1}\ \end{matrix}\right),$
where 
$$\phi_1(\zeta)=1 \ \text{and}\ \phi_k(\zeta)=\zeta^{k-1}+\zeta^{k-3}+ ... +\zeta^{-k+3}+\zeta^{-k+1}$$
for $k>1$.
Since $|\phi_k(\pm 1)|=2k-1$ and 
$\phi_k(\zeta)=\frac{\zeta^{k}-\zeta^{-k}}{\zeta-\zeta^{-1}}$,
for $\zeta\ne \pm 1,$
$$\rho_{\tau,\zeta}(b)^{2n+1}=\left(\begin{matrix} \zeta^{2n+1} & 0\\ \tau\cdot \phi_{2n+1}(\zeta) & \zeta^{-(2n+1)}\ \end{matrix}\right)= -{\rm I} = \rho_{\tau,\zeta}(a)^2$$
holds if and only if $\zeta$ is a $(2n+1)$-st root of $-1$, $\zeta\ne -1$ which will be required from now on. 
(The value of $\tau$ can be arbitrary.)
In other words, \eqref{e.torusrep} indeed defines an $SL(2,\C)$-representation of $\pi_1(E_K)$ for all $\tau$ and these values of $\zeta\ne -1.$
By a direct calculation, 
$$\Tr \rho_{\tau,\zeta}(aba^{-1}b^{-1})= 2 - \zeta^2 \tau (\tau +2i(\zeta-\zeta^{-1})).$$
Hence, $\rho_{\tau,\zeta}$ is irreducible if and only if
\begin{equation}\label{e.zirr}
\tau\ne 0, -2i(\zeta-\zeta^{-1}).
\end{equation}

\brem\label{r.rhoconj}
Since $\Tr \rho_{\tau,\zeta}(b)=\zeta+\zeta^{-1},$ the representations $\rho_{\tau,\zeta}$ and $\rho_{\tau',\zeta'}$
are conjugate only if either $\zeta'=\zeta$ or $\zeta' =\zeta^{-1}.$ Since 
$\Tr \rho_{\tau,\zeta}(ab)=i\zeta-i\zeta^{-1}+\tau$, we have
$\tau'=\tau$ in the first case and 
$$i\zeta'-i(\zeta')^{-1}+\tau' = i\zeta-i\zeta^{-1}+\tau$$ 
in the second case.
Note that these conditions are not only necessary but also sufficient. Indeed, since $\C[X(\la a,b\ra)]=\C[t_a,t_b,t_{ab}]$, two irreducible representations of the free group $\la a,b\ra$ are conjugate if their traces agree on $a$, $b$, and $ab.$
\erem

To avoid non-uniqueness of $\rho_{\tau,\zeta}$ let us require $Im\, \zeta> 0$ from now on. 
Note that $\zeta$ cannot be real since we assumed $\zeta\neq -1.$

 By \eqref{e.ml},
 \begin{equation}\label{e.tmrho}
 t_m(\rho_{\tau,\zeta})= \Tr\, \rho_{\tau,\zeta}(\m)=
 i \zeta^{-n} - i \zeta^n - \tau\frac{\zeta^n-\zeta^{-n}}{\zeta - \zeta^{-1}}.
 \end{equation}

Since $t_b(\rho_{\tau,\zeta})=\zeta+\zeta^{-1}$ and $t_m(\rho_{\tau,\zeta})$ determine $\zeta^{\pm 1}$ and $\tau$,
each irreducible representation $\rho$ of $\pi_1(E_K)$ is determined, up to conjugation, by 
$t_b(\rho)$ and $t_m(\rho)$.

Moreover, since $\rho(a^2)= -{\rm I}$, Eq. \eqref{e.ml} implies $\rho(\lo)=-\rho(\m^{-4n-2})$ and, hence, 
\begin{equation}\label{e.A-torus}
t_l= - T_{4n+2} (t_m),
\end{equation}
 where, as before, $T_k(X)$ is the $k$-th Chebyshev polynomial of the first type. 

Let us now fix coprime integers $p,q$. A representation $\rho$ of $\pi_1(E_K)$ that factors to a representation of $\pi_1(E_K(p/q))$ satisfies $\rho(m^p\cdot \ell^q)={\rm I}$. By \eqref{e.ml}, $\rho(\ell)=-\rho(m)^{-4n-2}$ for irreducible representations. Hence, the necessary and sufficient condition for $\rho_{\tau,\zeta}$ factoring to a representation of $\pi_1(E_K(p/q))$ is
\begin{equation}\label{e.rhompq}
\rho_{\tau,\zeta}(m)^{p-(4n+2) q} = (-1)^q\cdot I
\end{equation}

\begin{lemma}\label{l.pm2}
If \  $Tr\, \rho_{\tau,\zeta}(m)=\pm 2$, then $\rho_{\tau,\zeta}$ does not factor to a representation of $\pi_1(K(p/q)).$
\end{lemma}

\begin{proof}Suppose that $\rho_{\tau,\zeta}(m)$ has trace $\pm 2$ and satisfies \eqref{e.rhompq}. By putting $\rho_{\tau,\zeta}(m)$ in an upper triangular position we see that its eigenvalues must
both $1$ or $-1$. Furthermore, since its positive power is $\pm {\rm I}$, it must be $\pm {\rm I}$ itself. Hence, we obtain $\rho_{\tau,\zeta}(\m)= \pm {\rm I}$.
But by \eqref{e.ml},
$$\rho_{\tau,\zeta}(m)=
\left(\begin{matrix} i & 1\\ 0& -i\\ \end{matrix}\right)
\left(\begin{matrix} \zeta^{-n} & 0\\ -\tau\cdot \phi_n(\zeta) & \zeta^{n}\\ \end{matrix}\right)$$
has $\zeta^n$ in its upper right entry, and hence, it cannot be $\pm {\rm I}.$
\end{proof}
\smallskip

 Next recall  that $\rho_{\tau,\zeta}$ is reducible for $\tau=0$ and $-2i(\zeta-\zeta^{-1})$.

\begin{lemma}\label{l.except} 
(a) If $4\mid p$ then there are $\gcd{(2n+1,p)-1}$ reducible representations $\rho_{\tau,\zeta}$
factoring to a representation of $\pi_1(E_K(p/q))$.\\
(b) There are no such representations when $4$ does not divide $p$.
\end{lemma} 

\begin{proof} 
By Equations \ref{e.zirr} and \ref{e.tmrho}, $\rho_{\tau,\zeta}$ is reducible if and only if $Tr \rho_{\tau,\zeta}(m)=\pm i(\zeta^n-\zeta^{-n}).$ Note that this trace is not $\pm 2$ and, hence, $\rho_{\tau,\zeta}(\mu)$ is diagonalizable. Consequently, $\rho_{\tau,\zeta}(\mu)$ is conjugate to a diagonal matrix with eigenvalues $(i\zeta^n,-i\zeta^{-n})$ or $(i\zeta^{-n},-i\zeta^n).$ By \eqref{e.rhompq} $\rho_{\tau,\zeta}$ factors through a representation of $\pi_1(E_K(p/q))$ iff $\rho_{\tau,\zeta}(m)^{p-q(4n+2)}=(-1)^qI,$ which in our case is equivalent to $\rho_{\tau,\zeta}(m)^p=I,$ since $(i\zeta^{\pm n})^{4n+2}=-1.$

Therefore, if $\rho_{\tau,\zeta}$ factors through $\pi_1(E_K(p/q))$ then $i\zeta^{\pm n}$ is a $p$-th root of unity. However, since its order is always divisible by $4,$ hence $4$ must divide $p.$ If $p=4p',$ then $i\zeta^{\pm n}$ is a $4p'$-th root of unity for exactly $gcd(2n+1,p')-1=gcd(2n+1,p)-1$ choices of $\zeta\neq -1$ such that $\zeta^{2n+1}=-1.$

Finally note that the representations $\rho_{0, \zeta}$ and $\rho_{-2i(\zeta-\zeta^{-1}), \zeta^{-1}}$ are equivalent in the character variety because they have the same traces at $a,b$ and $ab.$ Consequently, there are $gcd(2n+1,p)-1$ distinct reducible representations factoring through a representation of $\pi_1(E_K(p/q))$ among the $\rho_{\tau,\zeta}.$
\end{proof}

We are now ready to prove  the following theorem that deals with reducedness of character varieties of 3-manifolds obtained by Dehn surgery on $(2,2n+1)$-torus knots.
\bthm\label{t.torusreduced} Let $p$ and $q$ be coprime, such that if $p$ is divisible by $4$ then it is coprime with $2n+1$. Then
$\cX(E_{T_{(2,2n+1)}}(p/q))$ is finite and reduced.
\ethm

\bpr 
The statement follows from Theorem \ref{t.reduced} after we check the assumptions of this reducibility criterion. Since $(2,2n+1)$-torus knots are $2$-bridge, $\cX(E_{T_{(2,2n+1)}})$ is reduced by \cite{Le06}.  Hence we may work with $X(E_{T_{(2,2n+1)}})$ instead of $\cX(E_{T_{(2,2n+1)}})$ in the proof. 

By the discussion above, the non-abelian components of $X(E_{T_{(2,2n+1)}})$ are indexed by the values of $\xi$ which is $(2n+1)$-st root of $-1$ with $Im\, \xi>0.$ Consequently, they do not intersect each other.
By the assumptions of the statement and by Lemma \ref{l.except}, the non-abelian irreducible components do not intersect the abelian one at points of $X(E_{T_{(2,2n+1)}}(p/q))$. Consequently,  each point of 
$X(E_{T_{(2,2n+1)}}(p/q))$ belongs to a unique irreducible component of $X(E_{T_{(2,2n+1)}}).$

As we discussed earlier, condition (a) of Theorem \ref{t.reduced} is satisfied for all abelian characters of $X(E_{T_{(2,2n+1)}}(p/q))$.

Let $U_\xi\subset X(E_{T_{(2,2n+1)}})$ be the set of characters $Tr\, \rho_{\tau, \xi}$ for all $\tau\ne 0, -2i(\zeta-\zeta^{-1}).$ It is an open set and each character of $X(E_{T_{(2,2n+1)}}(p/q))$ belongs to $U_\xi$ for some $\xi.$
Note that each point of $U_\xi$ is determined by $t_b$ and $t_m$. The first one is constant $\xi+\xi^{-1}$ on $U_\xi$ and the latter one is linear in $\tau$ (c.f. \eqref{e.tmrho}). Consequently, $r$ restricted to $U_\xi$ is an isomorphism onto its image. This shows that condition (a) of Theorem \ref{t.reduced} is satisfied.

By \cite{HS},
$$A_{(2,2n+1)}=1+\lambda\mu^{4n+2},$$
up to a multiplicative factor of a power of $M$, for $n\ne \pm 1$, and $A_{(2,2n+1)}=1$ otherwise.
(This also follows from our Eq. \ref{e.A-torus}.)
Note that $A_{(2,2n+1)}(x^{-q},x^p)=1+x^{p-(4n+2)q}$ in the first case.
Hence, $A_{(2,2n+1)}(x^{-q},x^p)$ has simple roots only and condition (b) of Theorem \ref{t.reduced} is satisfied.

Condition (c) of Theorem \ref{t.reduced} follows immediately from Lemma \ref{l.pm2}.
\epr

\subsection{Dimensions of skein modules under Dehn filling}
Next we will apply our results above to compute the skein modules $ S(M,\Q(A))$ for infinite families of 3-manifolds of the form
$M:= E_{(2,2n+1)}(p/q)$.  
Define
$$\tau_{n,p,q}:=|p/2-(2n+1) q| -\frac{1}{2}\delta_{2\nmid p},$$
where we recall that $\delta_{2\nmid p}=1$ when $p$ is odd, and $\delta_{2\nmid p}=0$ when $p$ is even.

\begin{theorem}\label{cor:BS} 
Let $p$ and $q$ be coprime, such that if $p$ is divisible by $4$ then it is coprime with $2n+1$.
 Let $M:= E_{T_{(2,2n+1)}}(p/q)$. If  $p/q\notin \lbrace 0,4n+2 \rbrace$, then
 $$\dim_{\Q(A)} S(M,\Q(A)) =|X(M)| =\tau_{n,p,q}\cdot n +  1+\left\lfloor\frac{|p|}{2}\right\rfloor.$$
 \end{theorem}
 
 Before we can proceed with the proof of the theorem we need some preparation:
Observe that $\Tr\, \rho(m)^k =T_k(t)$, for $t:=\Tr\, \rho(m)$.
Hence, Eq. \eqref{e.rhompq} implies
\begin{equation}\label{e.rhompqt}
T_k (t)= 2\cdot (-1)^q,  \ {\rm for} \  k=|p-(4n+2) q|.
\end{equation}

\def\cT{\mathcal T}

Given $n$ and coprime $p,q$ let $k=|p-(4n+2) q|$ and let $\cT_{n,p,q}$ denote the set of roots of
$P_k(t):=T_k(t)-2\cdot (-1)^q$ that are not equal to $\pm2$.

\blem
$$|\cT_{n,p,q}|=\tau_{n,p,q}=|p/2-(2n+1)q|-\frac{1}{2}\delta_{2 \nmid p}.$$  
\elem

\bpr
For $t=\mu+\mu^{-1}$,  we have $T_k(t)=\pm 2$ if and only if $\mu^k+\mu^{-k}=\pm 2$ if and only if $\mu^k=\pm 1$.
Since $t$ is preserved by the inversion of $\mu$, it follows that $P_k(t)$ has ${{k+1}\over 2}$ roots when $k$ is odd, and $\frac{k}{2}$ roots when $k$ is even.

Since $k$ has the same parity as $p,$ the polynomial $P_k$ has exactly one of $2$, $-2$ as its root, when $p$ is odd, and neither when $p$ is even. Hence, the statement follows.
\epr

Let ${\rm Z}_n=\{\zeta: \zeta^{2n+1}=-1,\ \zeta\ne -1,\ Im\, \zeta>0\},$ and note that $|{\rm Z}_n|=n$.
Next we give the proof of Theorem \ref{cor:BS}.

\begin{proof}[Proof of Theorem \ref{cor:BS}:] By Theorem \ref{t.torusreduced},
$\cX(E_{T_{(2,2n+1)}}(p/q))$ is reduced, and by Theorem \ref{t.fin_gen_skein}, $S(E_{T_{(2,2n+1)}}(p/q))$ is tame. Hence,
$\dim_{\Q(A)} S(M)=|X(M)|$, by Theorem \ref{t.main}.

It remains to be shown that $|X(M)|=\tau_{n,p,q}\cdot n +\left\lfloor\frac{|p|}{2}\right\rfloor+1$.
To that end, consider the map 
$$\Psi: {\rm Z}_n \times \cT_{n,p, q}\to X(E_{T_{(2,2n+1)}}),$$ 
sending $(\zeta,t)$ to $\rho_{\tau,\zeta}$ with $\tau$ determined by $\Tr\, \rho_{\tau,\zeta}(m)=t$. 
Note that, for a given $\zeta\in Z_n,$ the equation $t=\Tr\, \rho_{\tau,\zeta}(m)$ is satisfied for a unique value of $\tau$ by (\ref{e.tmrho}). Hence, $\Psi: {\rm Z}_n$ is well-defined.

We claim that the image of $\Psi$ lies in $X(M) \subset X(E_{T_{(2,2n+1)}}).$
Indeed, by the definition of $\cT_{n,p, q}$, the trace $t=\Tr\, \rho_{\tau,\zeta}(m)$ satisfies 
\eqref{e.rhompqt}. Furthermore, $\rho_{\tau,\zeta}(m)$ is diagonalizable, because $t\ne \pm 2.$ Therefore, it also satisfies \eqref{e.rhompq} (since its diagonal form does). 

By Lemma \ref{l.except}, all characters in its image are irreducible.
Furthermore, since each irreducible representation of $\pi_1(E_{T_{(2,2n+1)}})$ is conjugate to a $\rho_{\tau,\zeta}$ for some $\zeta\in {\rm Z}_n$ and some $\tau\in \C$, the map $\Psi$ is onto $X^{irr}(M)$. Indeed, by \eqref{e.rhompqt} and Lemma  \ref{l.pm2},
$\Tr \rho_{\tau,\zeta}(\m)$
is a root of $P_k(t)$ that is not equal to $\pm2$.

Since $\Psi$ is $1$-$1$ onto $X^{irr}(M)$,
$$|X^{irr}(M)|=|\cT_{n,p,q}| \cdot |{\rm Z}_n|=\tau_{n,p,q}\cdot n.$$
Since $H_1(M)=\Z/p,$ there are $\left\lfloor\frac{|p|}{2}\right\rfloor+1$ abelian
characters of $\pi_1(M)$ and, hence, the statement follows.
 \end{proof}
 
For $p,q,n>0$, the ${1\over {q}}$-surgery on the $(2,2n+1)$ knot yields a Brieskorn sphere $$M=\Sigma(2, 2n+1, 2(2n+1)q-1).$$
 
By \cite{BC06}, if $a,b,c>0$ are mutually coprime integers then the $SL(2,\C)$-Casson invariant of the Brieskorn sphere $\Sigma(a, b, c)$ is 
$$\lambda_{SL(2,\C)}(M)=|X^{irr}(M)|=\frac{(a-1)(b-1)(c-1)}{4}.$$
The discussion above implies that $\pi_1(M)$ has $\tau_{n,p,q} \cdot n$ irreducible representations.
That count agrees with the formula for \cite{BC06}.
 
 \subsection{ Bases of coordinate rings} 
 
 Next we give bases for $\C[X(E_K(1/q))]$, for  $K=T_{(2,2n+1)}$.

The fundamental group of $K=T_{2,2n+1}$ has a presentation
 $\pi_1(E_K)=\la a,b \ |\  a^2=b^{2n+1}\ra$, where
the meridian and the longitude are given by 
$m=a\cdot b^{-n}$ and $l=a^2\cdot m^{-4n-2}$, respectively.

For simplicity of notation, we
set $\tau_{n,q}:= \tau_{n,1,q}=\frac{|(4n+2) q-1|-1}{2}$ and we set $\cT_{n,q}:=\cT_{n,1,q}$.

\begin{theorem}\label{t.torusb} 
For every $q\in \Z$ and $M:=E_{T_{(2,2n+1)}}(1/q)$, there is a basis for $\C[X(M)]$ given by
 $B=\{ t_m^i t_b^j :  0\leq i<  \tau_{n,q} \ {\rm and}\ 0\leq j< n\}\cup \{t_m^{\tau_{n,q}}\}$.
\end{theorem} 

\begin{proof} By Theorem \ref{cor:BS},  $|X(M)|=\tau_{n,q}\cdot n+1$. Thus the cardinality of 
the set  $B$ 
coincides with that of $X(M)$. Therefore, it is enough to show that the above set is linearly independent.
Assume that 

$$\nu\cdot t_m^{\tau_{n,q}}+\Sigma_{i=0}^{\tau_{n,q}-1}\Sigma_{j=0}^{n-1} \lambda_{i,j} t_m^i t_b^j=0 \ \text{on}\ X(M),$$
for some $\nu, \lambda_{i,j}\in \C.$ 
By restricting this identity to the set $X_{t_0}(M)$ of classes of representations $\rho$ such that $\Tr\, \rho(m)=t_0,$ for some fixed $t_0\in \C,$ we can think of the left side of the equation above as 
a polynomial in $t_b.$ Assume $t_0\in \cT_{n,q}$ now.
By the proof of Theorem \ref{cor:BS},
for every $\zeta\in {\rm Z}_n$ there is $\tau\in \C$ such that the character of $\rho_{\tau,\zeta}$ is in $X_{t_0}(M)$. There are $n$ distinct values 
$$t_b(\rho_{\tau,\zeta})=\rho_{x,\zeta}(b)=\zeta+\zeta^{-1}$$ 
for these representations.
Hence, by the Vandermonde determinant argument of  Proposition \ref{t.basis-char},
the powers $t_b^0,...,t_b^{n-1}$ are linearly independent on $X_{t_0}(M).$ Consequently, for  every $1\leq j< n$, we have
\begin{equation}\label{e.lambdas1}
\Sigma_{i=0}^{\t-1} \beta_{i,j}\cdot  t_0^i=0.
\end{equation} 
Since Equation (\ref{e.lambdas1}) is satisfied for any element $t_0\in \cT_{n,q}$, a Vandermonde  determinant argument implies that $\beta_{i,j}=0$ for all $j>0$ and all $i$.

It remains to be shown that $\alpha=0$ and $\beta_{0,j}=0$ for all $i$.
Since $t_m=t_0$ on $X_{t_0}(M)$, 
$$\alpha \cdot t_0^{\t}+\Sigma_{i=0}^{\t-1} \beta_{i,0} \cdot t_0^i=0$$ 
for every $t_0\in \cT_{n, q}$.

Hence, by the Vandermonde determinant argument, it is enough to show that $t_m$ takes $\t+1$ distinct values. That is indeed the case: By the argument in the proof of Theorem \ref{cor:BS}, it takes all values of $\cT_{n,q}$. In addition,  it also takes value $2$ (which is not in $\cT_{n,q}$), because of the trivial representation.
\end{proof}

By Theorem \ref{t.red-tame} and  Proposition \ref{p.basis} we obtain the following:

\begin{theorem}\label{t.bases-i}
Let $K$ and $B$ be as in Theorems  \ref{t.torusb}  or \ref{t.basisf8} and let $M=E_K(1/q)$.
Then any collection of framed unoriented links in
$M$ that are mapped bijectively to $B$ under the isomorphism $\psi: S_{-1}(M)\to \C[\cX(M)]$  forms a
basis of $S(M)$.
\end{theorem}


\section{On the non-triviality  of $S(M,\Q(A))$}
\label{sec:non_triviality}
The goal of this section is to prove Theorem \ref{thm:nontriviality}  stated in the Introduction and generalize it to 3-manifolds with boundary.

We start by reviewing the definition of Gilmer-Masbaum's evaluation map, defined in \cite{GM19}. Let $\mathbb{U}\subset \C$ be the set of roots of unity of order $2N$ with $N\geq 3$ odd. Let  $\C^{\mathbb{U}}_{a.e.}$  denote the set of functions $ \C$-valued functions that are defined on all but finitely elements of $\mathbb{U}$ up to the following equivalence relation: two functions are equivalent if they coincide on all but finitely many points of $\mathbb{U}$. The set $\C^{\mathbb{U}}_{a.e.}$ becomes a  $\Q(A)$-vector space, by defining
$$R\cdot f : \zeta \in \mathbb{U} \rightarrow R(\zeta)f(\zeta),$$ 
for any $R:=R(A)\in \Q(A)$ and $f\in \C^{\mathbb{U}}_{a.e.}$. Note that this operation is well defined, since any $R(A)\in \Q(A)$  has finitely many poles and hence can be evaluated at all but finitely many roots of unity. 

Recall from Section \ref{sec:equiv_maps}, that for a framed link $L$ in a $3$-manifold $M$ and a primitive $2N$-th root of unity $\zeta$ with $N\geq 3$ odd, $RT^{\zeta}(M,L)$ denotes the $\mathrm{SO}(3)$-Reshetikhin-Turaev invariant of the pair $(M,L)$ at  $\zeta$. 

\begin{theorem}\label{thm:evaluation_map}\cite[Theorem 2]{GM19} There is  $\Q(A)$-linear map
$$ev: S(M,\Q(A)) \longrightarrow \C_{a.e.}^{\mathbb{U}},$$
defined by sending 
a framed link $L$ in $M$ to the function $\zeta \in \mathbb{U} \longrightarrow RT^{\zeta}(M,L) \in \C$.\end{theorem}
Note that a link $L$ is mapped to $0 \in \C_{a.e.}^{\mathbb{U}}$ if and only if $RT^{\zeta}(M,L)=0$ for all but finitely many primitive roots of unity $\zeta$ of order $2N$ with $N$ odd.

Let $M$ be a rational homology sphere and $|H_1(M,\Z)|$ be the number of elements of $H_1(M,\Z)$. 
Let $\left( \dfrac{a}{b} \right)$ be the Legendre symbol for $a,b \in \N$.
For the proof of Theorem  \ref{thm:nontriviality} we will use following theorem of Murakami:
 
 \begin{theorem}\cite{Mur95} \label{thm:SO(3)inv_RHS} Let $M$ be a rational homology sphere and
 let $h_1=|H_1(M,\Z)|.$ For any odd prime $p$ such that $p\nmid h_1$ and any primitive $2p$-th root of unity $\zeta$, we have
$$h_1 RT^{\zeta}(M) \in \Z[\zeta^2]\ \ {\rm and} \ \
h_1 RT^{\zeta}(M)= \left( \dfrac{h_1}{p} \right) \ \textrm{mod} \ \zeta^2-1$$
\end{theorem}

We are now ready to prove Theorem  \ref{thm:nontriviality} of the Introduction asserting that the empty link does not vanish in $S(M,\Q(A))$ for rational homology spheres $M$.

\begin{proof}[Proof of Theorem \ref{thm:nontriviality}] 
Theorem \ref{thm:SO(3)inv_RHS} implies that $RT^{\zeta}(M)$ is non-zero for all $\zeta$ of order $2p$ for sufficiently large primes $p$. Since there are infinitely many of them, $ev(\emptyset) = RT^{\zeta}(M)\ne 0$ in $\C^{\mathbb{U}}_{a.e.}$
\end{proof}

Theorem \ref{thm:nontriviality} generalizes to 3-manifolds with boundary as follows:
\begin{corollary}\label{cor:Kauffmanbracket2} Suppose that  $M$ is  a $3$-manifold  with boundary whose first rational homology is carried by its boundary. That is,  the map 
$H_1(\partial M,\Q)\to H_1(M,\Q)$ induced by inclusion 
 is onto. Then, $\dim_{\Q(A)}S(M,\Q(A)) \geq 1$.
\end{corollary}

\begin{proof}

For any disjoint union of handlebodies $H$ with boundary $\partial H\simeq \partial M=\Sigma$,
and for $\bar M=M\, \underset{\Sigma}{\bigcup}\, H$, the embedding $M\hookrightarrow \bar M$ induces an epimorphism $S(M)\to S(\bar M)$. Therefore, by Theorem \ref{thm:nontriviality}, it is enough to show that $\bar M$ is a rational homology sphere for some $H$.

By the Mayer-Vietoris sequence, 
$$H_1(\bar M,\Q)=H_1(\Sigma,\Q)/(K_M+K_H),$$ where 
$$K_M=\mathrm{Ker}\, (H_1(\Sigma,\Q) \longrightarrow H_1(M,\Q)) \ \text{and}\ K_H=\mathrm{Ker}\, (H_1(\Sigma,\Q)\longrightarrow H_1(H,\Q)).$$
It remains to be shown that $K_M+K_H=H_1(\Sigma,\Q)$ for some handlebody $H$ as above. Note that the intersection form on $H_1(\Sigma,\Q)$ is a symplectic and $K_M$ and $K_H$ are Lagrangian subspaces of $H_1(\Sigma,\Q)$ by the Poincar\'e duality. Hence it is enough to show that
\begin{equation}\label{e.comp}
K_M\cap K_H=0.
\end{equation}
Note that the symplectic space $H_1(\Sigma,\Q)$ is the direct sum of symplectic spaces $H_1(\Sigma_i,\Q)$ over all connected components $\Sigma_1,..., \Sigma_r$ of $\Sigma.$ Since any Lagrangian of $H_1(\Sigma_i,\Q)$ is realized as a $K_{H_i}$ for a handlebody with boundary $\Sigma_i,$ it suffices to prove that $K_M$ is transverse to a Lagrangian of $H_1(\Sigma,\Q)$ which is a sum of Lagrangians in each $H_1(\Sigma_i,\Q).$ Hence,
 the statement follows from the lemma below. 
 \end{proof}

\begin{lemma}\label{l.sumLagrangians}Let $V_1,\ldots,V_k$ be finite dimensional symplectic vector spaces over $\Q$, and $L$ be a Lagrangian of $V=\underset{1\leq i \leq k}{\bigoplus} V_i.$ Then there exists a collection of Lagrangians $L_i$ of $V_i,$ such that $L$ is transverse to $\underset{1\leq i \leq k}{\bigoplus} V_i.$
\end{lemma}

\begin{proof}
Let $a_1,b_1,\ldots,a_n,b_n$ be a symplectic basis of $V$, i.e. the symplectic form on all pairs of them vanishes, except for $\omega(a_1,b_1)= ... = \omega(a_n,b_n)= 1$.  We assume additionally that the above basis is compatible with the decomposition. We will prove by induction on $0\leq k\leq n$ that $L$ is transverse to an isotropic subspace $L_k$ which is spanned by one of the vectors in each of the pairs $\lbrace a_1,b_1\rbrace, \ldots, \lbrace a_k,b_k\rbrace.$

This fact is obvious for $k=0.$ If such a space $L_k$ is constructed, assume that by contradiction that $L$ contains both a vector of the form $a_{k+1}+v$ and one of the form $b_{k+1}+v'$ with $v,v'\in L_k.$ Then since $L_k$ is isotropic, we have $\omega(a_{k+1}+v,b_{k+1}+v')=1,$ which contradicts the fact that $L$ is a Lagangian. So assume for example that $a_{k+1}+v \notin L$ for all $v\in L_k.$ Then $L_{k+1}\cap L=\lbrace 0 \rbrace,$ where $L_{k+1}=L_k \oplus \Q a_{k+1}.$
\end{proof}

\begin{remark} The generalized Chen-Yang volume volume conjecture \cite{CY18} and its simplicial volume generalization stated by Detcherry-Kalfagianni  in \cite[Conjecture 8.1]{DK20}, and by Detcherry-Kalfagianni-Yang for link complements \cite{DKY18}, imply that for any 3-manifold $M$ whose JSJ decomposition contains hyperbolic pieces we have $RT^{\zeta}(M)\neq 0$ for infinitely many roots of unity $\zeta$. As noted in Sec. \ref{sec:non_triviality}, the later statement implies that the $\emptyset$ link in $M$ represents a non-zero element in $S(M,\Q(A))$ and, hence, $\dim_{\Q(A)}S(M,\Q(A)) \geq 1$. Consequently, the simplicial volume conjecture of \cite{DK20} implies Conjecture \ref{question:nontriviality} for 3-manifolds with at least one hyperbolic piece in their JSJ decomposition.
\end{remark}


\section{Open questions and remarks}
\label{sec:questions}

\subsection{Relation to Abouzaid-Manolescu homology}

In \cite[Conjecture D]{GS23} and \cite[Section 6.3]{GJS19}, it is conjectured that $\dim_{\Q(A)} S(M)$ is equal to
the dimension of the zero degree part of the  Abouzaid-Manolescu 
homology $HP^{\bullet}_{\#}(M)$, \cite{AM20}.  By the following result, our Theorem \ref{t.main1-i} verifies the conjecture 
of \cite{GJS19} for $3$-dimensional homology spheres $M$ with tame 
$S(M, \Q[A^{\pm 1}])$ and $\cX(M)$ finite and reduced.

\begin{proposition}\label{p.HP}
For $3$-dimensional $\Z$-homology spheres $M$ with $\cX(M)$ finite and reduced, the zero degree part of $HP^{\bullet}_{\#}(M)$ has dimension $|X(M)|$.
\end{proposition}

\begin{proof}
Firstly, we claim that under the above assumptions, the representation scheme $\cR(M)$ of $M$ is regular.
The connected component of $\cR(M)$  of the trivial representation is a single point and, since $\cX(M)$ is reduced, that trivial component is a single, reduced point, and hence regular. Since all other representations of $\pi_1(M)$ are irreducible, all other points of $\cR(M)$ are regular by \cite[Lemma 2.4]{AM20}.
$HP^{\bullet}_{\#}(M)$ is the homology of certain sheaf $P^\bullet_\#(M)$ over the representation scheme $\cR(M).$
By \cite[Theorem 1.4]{AM20},  $P^\bullet_\#(M)$ is given by a local system on $R(M).$ That system is obviously trivial over the trivial representation component. It is also trivial over other components of $R(M)$ by \cite[Lemma 8.3]{AM20}. Hence, the statement follows.
\end{proof} 

The next  proposition verifies \cite[Conjecture D]{GS23}  for the  families of 3-manifolds of Theorem  \ref{t.dimensions}. Note that these families also contain infinitely many examples that are not
 $\Z$-homology spheres.
 
 \begin{proposition}\label{p.moreexamples}  The zero degree part of $HP^{\bullet}_{\#}(E_{K}(p/q))$ has dimension
 $\dim_{\Q(A)} S(E_{K}(p/q))$, for
 (a) $K=4_1$ and for all but finitely many $p/q$, including all slopes with $p=1$ \\
(b) $K=T_{(2,2n+1)}$ and all $n\in \Z$ and all slopes $p/q\notin \lbrace 0, 4n+2\rbrace$, where $p$ is either not divisible by $4$ or coprime with $2n+1$.
\end{proposition}
 \begin{proof} In \cite[Theorem 2]{Examples}, Neithalath  computes $HP^{\bullet}_{\#}(E_{K}(p/q))$ under the hypotheses  that
$\cX(E_{K}(p/q)$ is reduced and finite and the Alexander polynomial of $K$ has no roots that that roots of unity of order $p$ if $p$ is odd, or of order $p/2$ if $p$ is even.

By Theorem \ref{t.red-tame}, $\cX(E_{K}(p/q))$ is reduced and finite. The Alexander polynomial of $K=4_1$ is $\Delta_{4_1}(t)=1-3t+t^2$ which has no roots that are roots of unity. Moreover, the roots of the Alexander polynomial of $K=T_{(2,2n+1)},$
$$\Delta_{T_{2,2n+1}}(t)=\frac{(t^{4n+2}-1)(t-1)}{(t^{2n+1}-1)(t^2-1)}$$ are all $4n+2$-th roots of unity of even order except $ -1.$ It follows that if $K$ and $p/q$ are as in (a) or (b) of the statement of the proposition, then the hypotheses of
 \cite[Theorem 2]{Examples} are satisfied. 

By \cite[Theorem 2]{Examples} and its proof,  the dimension of the 0-th degree of $HP^{\bullet}_{\#}(E_{K}(p/q))$
is equal to
$$\lambda_{SL(2,C)}(E_{4_1}(p/q))+1+\left \lfloor\frac{|p|}{2}\right \rfloor=|X^{irr}(E_{K}(p/q))|+1+\left \lfloor\frac{|p|}{2}\right \rfloor.$$
On the other hand, by Theorem \ref{t.X41}  and \ref{cor:BS} , and their proofs, 
the quantity above is also equal to $\dim_{\Q(A)} S(E_{K}(p/q))$. 
\end{proof}

\subsection{Questions} The hypothesis of Theorem \ref{t.main} implies that the character variety $X(M)$ is finite. On the other hand, while $X(\R P^3 \sharp \R P^3)$ is finite, \cite[Proposition 4.19]{Mro11} shows that the skein module $S(\R P^3 \sharp \R P^3, \Q[A^{\pm 1}])$ does not split as a sum of cyclic  $\Q[A^{\pm 1}]$-modules and, hence, it is not tame. However, we make the following conjecture, which implies that manifolds with tame skein module are abundant:

\begin{conjecture}\label{c.non-Haken} Let $M$ be a closed $3$-manifold that is irreducible and contains no incompressible surface. Then $S(M)$ is finitely generated over $\Q[A^{\pm 1}]$.
\end{conjecture}
The motivation behind Question \ref{c.non-Haken} is that by Culler-Shalen theory \cite{CullerShalen}, if $M$ contains no incompressible surface then $X(M)$ is finite. 
In \cite{Seifert} we have proved Conjecture \ref{c.non-Haken}  for Seifert fibered 3-manifolds.

Note that Przytycki's Conjecture E in Problem 1.92 of \cite{Kirby} postulates that $S(M,\Z[A^{\pm 1}])$ is free for closed, non-Haken $M$. In particular, $S(M,\Z[A^{\pm 1}])$ is finitely generated by the Gunningham-Jordan-Safronov finiteness theorem, \cite{GJS19}. Hence, our Conjecture \ref{c.non-Haken} is a consequence of that of Przytycki.
 
Let us now discuss a possible generalization of Theorem \ref{t.main1-i} to the case of infinite $X(M)$. Since $S(M,\Q(A))$ is always finite dimensional by \cite{GJS19}, the inequality of Theorem \ref{t.main1-i} cannot hold in this case. To propose a possible replacement for the right hand side of the inequality, let us call isolated, reduced characters in $\cX(M)$ \underline{rigid} and denote their set by 
$X(M)^{rig}$.
 (Note that it is a finite subset of $X(M)$.)
 
\begin{conjecture}\label{conjecture:lower_bounds_dimension}  For every closed $3$-manifold 

$$\dim_{\Q(A)}S(M,\Q(A)) \geq |X(M)^{rig}|.$$
\end{conjecture}

We conclude the section with a stronger version of Conjecture \ref{question:nontriviality} for prime 3-manifolds.
 Note that it was proved by Przytycki \cite{Prz00} that for a connected sum $M_1 \# M_2$ one has $$S(M_1 \# M_2,\Q(A)) \simeq S(M_1,\Q(A))\otimes S(M_2,\Q(A))$$
Therefore, the set of closed manifolds such that $\dim_{\Q(A)} S(M,\Q(A))=1$ is closed under connected sum. Moreover, $S(S^3)\simeq \Z[A^{\pm 1}]$ and $S(S^2\times S^1,\Q(A))\simeq \Q(A)$, cf. \cite{HP95}.

\begin{question}\label{question:dim2} Is $\dim_{\Q(A)} S(M,\Q(A))\geq 2$ for all prime, closed connected $3$-manifolds $M$ other than $S^3$ and $S^2\times S^1$?
\end{question}

\bibliographystyle{hamsalpha}
\bibliography{biblio}
\end{document}